\documentclass[reqno,12pt,twosided,a4paper]{amsart}
\usepackage{amssymb,amscd,stmaryrd}
\usepackage[top=3cm,bottom=3cm,left=2.5cm,right=2.5cm,footskip=17mm]{geometry}
\usepackage[mathcal]{eucal}
\usepackage{array,float}
\usepackage{enumitem}
\usepackage{amsmath}
\usepackage{amssymb, xy}
\input xy
\xyoption{all}
\usepackage{pdflscape}
\usepackage{hyperref}
\usepackage[draft]{graphicx}

\usepackage{xcolor}
\usepackage{color}
\numberwithin{equation}{section}

\newtheorem{theorem}{Theorem}[section]
\newtheorem{lemma}[theorem]{Lemma}
\newtheorem{proposition}[theorem]{Proposition}

\theoremstyle{definition}
\newtheorem{definition}[theorem]{Definition}

\newtheorem{remark}[theorem]{Remark}

\newcommand{\ot}{\otimes}
\newcommand{\C}{\mathcal{C}}

\newcommand{\one}{\textbf{1}}

\newcommand{\wb}{\overline}
\newcommand{\Z}{\mathbb{Z}}
\newcommand{\Ga}{\Gamma}
\newcommand{\YD}{\mathcal{YD}}
\newcommand{\GGYD}{_G^G\YD}

\newcommand{\B}{\text{B}}
\newcommand{\Ow}{\mathcal{O}}

\newcommand{\A}{\mathbb{A}}
\newcommand{\lt}{\leadsto}
\newcommand{\Gg}{\mathbb{G}}

\newcommand{\Ker}{\text{Ker}}

\newcommand{\Hom}{\text{Hom}}
\newcommand{\Aut}{\text{Aut}}
\newcommand{\GL}{\text{GL}}
\newcommand{\Pp}{\text{P}}
\newcommand{\la}{\lambda}
\renewcommand{\Im}{\text{Im}}

\begin{document}
\title{Geometric perspective on Nichols algebras}
\author{Ehud Meir}
\address{Institute of Mathematics, University of Aberdeen, Fraser Noble Building, Aberdeen AB24 3UE, UK}
  \email{meirehud@gmail.com}

\maketitle
\begin{abstract} 
We formulate the generation of finite dimensional pointed Hopf algebras by group-like elements and skew-primitives in geometric terms.
This is done through a more general study of connected and coconnected Hopf algebras inside a braided fusion category $\C$. 
We describe such Hopf algebras as orbits for the action of a reductive group on an affine variety. We then show that the closed orbits are precisely the orbits of Nichols algebras, and that all other algebras are therefore deformations of Nichols algebras. 
For the case where the category $\C$ is the category $\GGYD$ of Yetter-Drinfeld modules over a finite group $G$, this reduces the question of generation by group-like elements and skew-primitives to a geometric question about rigidity of orbits. 
Comparing the results of Angiono Kochetov and Mastnak, this gives a new proof for the generation of finite dimensional pointed Hopf algebras with abelian groups of group-like elements by skew-primitives and group-like elements.
We show that if $V$ is a simple object in $\C$ and $\B(V)$ is finite dimensional, then $\B(V)$ must be rigid. 
We also show that a non-rigid Nichols algebra can always be deformed to a pre-Nichols algebra or a post-Nichols algebra which is isomorphic to the Nichols algebra as an object of the category $\C$. 
\end{abstract}

\begin{section}{Introduction}
One of the fundamental problems in the theory of finite dimensional pointed Hopf algebras is to determine if such algebras are generated by group-like elements and skew-primitives. This aims to generalize the following classical classification result of Cartier, Milnor, Moore, and Kostant from the 1960s:
\begin{theorem}[Cartier-Milnor-Moore-Kostant,60s]
Let $H$ be a cocommutative Hopf algebra over an algebraically closed field $K$ of characteristic zero.
Then $H$ is the crossed direct product of a group algebra with the universal enveloping algebra of a Lie algebra.
In particular, $H$ is generated by group-like elements and primitive elements.
\end{theorem}

This problem was studied thoroughly in case the group of group-like elements in the Hopf algebra is abelian and the ground field has characteristic zero.  
In \cite{AS10} Andruskiewitsch and Schneider proved that such a Hopf algebra must be generated by group-like elements and skew-primitives, and gave a complete classification of such algebras in case the group of group-like elements does not have prime divisors which are $\leq 7$. This was done by the \textit{lifting method} and by a deep study of the structure of the possible \textit{Nichols algebras} arising in the category of Yetter-Drinfeld modules over an abelian group. The Nichols algebras correspond to the universal enveloping algebras in the above theorem. In \cite{Heckenberger} Heckenberger classified all Yetter-Drinfeld modules $V$ for which the Nichols algebra $\B(V)$ is finite dimensional, as part of a wider classification of Nichols algebras with finite root systems. In \cite{A1} Angiono described these Nichols algebras explicitly in terms of generators and relations and proved that all finite dimensional connected graded Hopf algebras in the category of Yetter-Drinfeld modules over a finite abelian group are Nichols algebras. 
Using the above results Angiono and Garcia-Igelsias gave in \cite{AG} a complete classification of finite dimensional pointed Hopf algebras with abelian groups of group-like elements. For more classification results, including the case where the group of group-like elements is non-abelian, see \cite{AnSa19} and the survey \cite{And14}.

The starting point of these classification results is the following: let $G$ be a finite group and let $H$ be a finite dimensional pointed Hopf algebra over an algebraically closed field $K$ of characteristic zero, whose group of group-like elements is isomorphic to $G$. The fact that $H$ is pointed implies that the \textit{coradical filtration} $H_0=KG\subseteq H_1\subseteq\cdots\subseteq H_n=H$ of $H$ is a Hopf algebra filtration. This implies that $\text{gr}H:=\oplus_{i=1}^{n} H_i/H_{i-1}$ is a graded Hopf algebra. Moreover, if $\text{gr}H$ is generated by group-like elements and skew-primitives then $H$ is generated by group-likes and skew-primitives as well.
The inclusion $H_0\to \text{gr}H$ splits, and by a result of Radford we can write $\text{gr}H \cong R\# KG$ where $R$ is a graded Hopf algebra in the category of Yetter-Drinfeld modules over $G$, $\GGYD$. 
The comultiplication of $R$ as a Hopf algebra in $\GGYD$ is different from the comultiplication of elements of $R$ in the Hopf algebra $H$. The skew-primitive elements become primitive elements in $R$. The original question then boils down to whether or not $R$ is generated by primitive elements, and not just skew-primitives.

In \cite{AS10,AS00} Andruskiewitsch and Schneider studied both the Hopf algebra $R$ and the dual Hopf algebra $R^*$, proved that finite dimensionality implies that certain relations among the elements of these Hopf algebras must hold, and concluded that both $R$ and $R^*$ are generated by primitive elements. This means that the algebra $R$ is in fact the \textit{Nichols algebra} $\B(V)$ where $V=\Pp(R)$ is the set of primitive elements in $R$. 

Andruskiewitsch and Schneider then also address the questions of the reconstruction of the original algebra $H$ out of $\text{gr}H$, and for what objects $V$ of $\GGYD$ the algebra $\B(V)$ is finite dimensional. 
The key-point in proving that any finite dimensional pointed Hopf algebra is generated by group-like elements and skew-primitives is to prove that all Hopf algebras $R$ in $\GGYD$ arising from the above construction are Nichols algebras.

In this paper we study a more general problem by using a different, geometric, method. 
Instead of looking at the category $\GGYD$ we look at a general braided fusion category $\C$. 
For every object $B$ in $\C$ we will construct an affine variety $X_B$, whose points represent the structure constants of connected coconnected Hopf algebras (these notations will be explained in Sections \ref{sec:preliminaries} and \ref{sec:variety}). The group $\Ga_B:=\Aut_{\C}(B)$ acts on $X_B$, and the orbits correspond to isomorphism types of Hopf algebras.  We will prove the following:
\begin{theorem}\label{thm:main1}
Let $R$ be a connected coconnected finite dimensional Hopf algebra in $\C$ such that $R\cong B$ as objects of $\C$. The orbit $\Ow_R$ of $R\in X_B$ is closed if and only if the algebra $R$ is isomorphic to a Nichols algebras.
In particular, all the orbits of $\Ga_B$ in $X_B$ are closed if and only if all the connected and coconnected Hopf algebras in $\C$ that are isomorphic to $B$ as objects of $\C$ are Nichols algebras.
\end{theorem}
If $R_1$ and $R_2$ are two algebras in $X_B$, we say that $R_1$ \emph{specializes} to $R_2$ if $\Ow_{R_2}\subseteq \overline{\Ow_{R_1}}$.
We also say in this case that $R_1$ is a \emph{deformation} of $R_2$. 
It is known that, for the action of an algebraic group on an affine variety, every orbit contains a closed orbit in its closure. The theorem above thus implies that every algebra in $X_B$ is a deformation of a Nichols algebra.

We thus focus our attention on studying deformations of finite dimensional Nichols algebras $\B(V)$, as such deformations are the possible obstructions to the generation by skew-primitives and the coradical (see Theorem \ref{thm:main2}).
\begin{definition}\label{def:rigidity}
The Hopf algebra $R\in X_B$ is called rigid 
if $R\in \overline{\Ow_{R'}}$ for some $R'\in X_B$ implies that $R\cong R'$.
\end{definition}
The ultimate goal will thus be to prove that $\B(V)$ is rigid whenever it is finite dimensional, as this will imply that all the algebras in $X_B$ are Nichols algebras and are therefore generated by primitive elements. 
We will prove the following result:
\begin{theorem}\label{thm:main3}
Assume that $V$ is simple in $\C$ and that $\B(V)$ is finite dimensional.
Then $\B(V)$ is rigid. 
\end{theorem}

Since our aim is to prove that all orbits in $X_B$ are closed, it is worthwhile asking how do hypothetical non-closed orbits in $X_B$ look like. To state the next result, recall that a \textit{pre-Nichols algebra} in a braided monoidal category $\C$ is a graded Hopf algebra in $\C$ which is generated by primitive elements (though not all the primitive elements are necessarily of degree 1). Thus, a pre-Nichols algebra $R$ is a quotient of the Hopf algebra $T(V)$ for some $V\in \C$ which also projects onto the Nichols algebra $\B(V)$. 
Dually, a post-Nichols algebra is a Hopf subalgebra of the graded-dual Hopf algebra of $T(V)$ that contains $\B(V^*)$. Post- and pre-Nichols algebras are graded-dual to each other (see Section 2.3 of \cite{AAB}).
\begin{theorem}\label{thm:main4}
Assume that $\B(V)$ is finite dimensional and not rigid. 
Then there is either a finite dimensional pre-Nichols algebra $R$ such that $\Pp(R)=V'\subsetneq V$ and such that $\B(V)\in \wb{\Ow_R}$, or a finite dimensional pre-Nichols algebra $R$ such that $\Pp(R)=V'\subsetneq V^*$ and such that $\B(V^*)\in \wb{\Ow_R}$. 
\end{theorem}
Summarizing Theorem \ref{thm:main1} and \ref{thm:main4}, we get the following result:
\begin{theorem}\label{thm:main5}
Let $\C$ be a braided fusion category. The following conditions are equivalent:
\begin{enumerate}
\item For every object $B\in \C$, all the orbits of the action of $\Ga_B=\Aut_{\C}(B)$ on $X_B$ are closed.
\item All finite dimensional Nichols algebras in $\C$ are rigid. 
\item All finite dimensional pre-Nichols algebras in $\C$ are Nichols algebras.
\item Every connected and coconnected Hopf algebra $R$ in $\C$ is isomorphic to $\B(\Pp(R))$.
\end{enumerate}
\end{theorem}
In case the category $\C$ is the category $_A^A\YD$ of Yetter-Drinfeld modules over a finite dimensional semisimple Hopf algebra $A$, the bosonization process, which produces from a Hopf algebra $R$ in $\C$ a Hopf algebra $R\# A$ in $Vec_K$, gives the following result: 
\begin{theorem}\label{thm:main2}
Let $A$ be a finite dimensional semisimple Hopf algebra. The following are equivalent:
\begin{enumerate}
\item For every $B\in _A^A\YD$ the orbits of $\Ga_B$ in $X_B$ are closed.
\item All finite dimensional Nichols algebras in $_A^A\YD$ are rigid.
\item Every Hopf algebra $H$ in which the coradical is a Hopf algebra isomorphic to $A$ is generated by the zeroth and first levels of its coradical filtration.
\item Every connected and coconnected Hopf algebra $R$ in $_A^A\YD$ is isomorphic to a Nichols algebra.
\end{enumerate}
For $A=KG$ where $G$ is a finite group the second statement says that every finite dimensional pointed Hopf algebra $H$ with $G(H)=G$ is generated by group-like elements and skew-primitives.
\end{theorem}
The study of deformations of Hopf algebras was initiated by Gerstenhaber and Schack in \cite{GS}. Du, Chen and Ye studied deformations of graded Hopf algebras in \cite{DCY}. Angiono, Kochetov and Mastnak studied deformations of Nichols algebras in \cite{AKM}. Deformations were also studied by Makhlouf in \cite{Makhlouf}. 
The deformations in the above papers are deformations by a parameter $\la$. 
We will show that our notion of rigidity, at least for Nichols algebras, is equivalent to the rigidity of Angiono, Kochetov and Mastnak. In \cite{AKM} the authors gave a proof that all Nichols algebras of diagonal type are rigid. Theorem \ref{thm:main2} above provides a new proof for the generation of pointed Hopf algebras with an abelian group of group-like elements by skew-primitives and group-likes.
In Section 4 of \cite{A1} Angiono proved this result by ruling out the existence of finite dimensional pre-Nichols algebras which are not Nichols algebras. The proof in this paper follows from the rigidity result of \cite{AKM} which is based on the description of Nichols algebras from \cite{A1} by generators and relations, but not on the case by case study done in Section 4 of \cite{A1}. 

This paper is organized as follows: in Section \ref{sec:preliminaries} we will give preliminaries about braided fusion categories, Hopf algebras, and the results from the theory of algebraic groups and geometric invariant theory which we will use here. In section \ref{sec:hopffusion} we will discuss in more detail Hopf algebras in braided fusion categories, and prove the equivalence of the second and third conditions of Theorem \ref{thm:main2}. In Section \ref{sec:catslinear} we will give a description of braided fusion categories using vector spaces and linear algebra. This will be used in Section \ref{sec:variety} to show that the collection of all connected and coconnected Hopf algebras which are isomorphic to a given object $B$ of $\C$ form an affine variety $X_B$. We will also construct an action of $\Ga_B:=\Aut_{\C}(B)$ on this variety, and show that the orbits correspond to isomorphism classes of Hopf algebras. In the end of Section \ref{sec:variety} we will also give a proof of Theorem \ref{thm:main3}. 
In Section \ref{sec:filtrations} we will discuss filtrations of Hopf algebras and their relation to geometric invariant theory. In Section \ref{sec:proofthm1} we will give a proof of Theorem \ref{thm:main1} and \ref{thm:main4}. In Section \ref{sec:rigidity} we will explain the relations between the different notions of rigidity and give a new proof, using the results of Angiono Kochetov and Mastnak, to the generation of pointed Hopf algebras with abelian group of group-like elements by group-like elements and skew-primitives. 
\end{section}

\begin{section}{Preliminaries}\label{sec:preliminaries}
\subsection{Braided fusion categories} We recall here briefly the definitions of fusion and braided categories. (see \cite{ENO}).
As stated in the introduction, we assume throughout the paper that our ground field $K$ is algebraically closed and of characteristic zero.
\begin{definition} A fusion category $\C$ over $K$ is an abelian category that satisfies the following properties:
\begin{enumerate}
\item The category $\C$ is enriched over $Vec_K$. This means that all hom-spaces in $\C$ are finite dimensional $K$-vector spaces.
\item The category $\C$ is semisimple. This means that every object in $\C$ can be written uniquely as a direct sum of simple objects.
\item The category $\C$ is monoidal. This means that we have a functor \begin{equation}\bigotimes: \C\times\C\to \C\end{equation} together with associativity isomorphisms \begin{equation}\alpha_{X,Y,Z}:(X\ot Y)\ot Z\to X\ot (Y\ot Z)\end{equation} for every three objects $X,Y,Z$ of $\C$ satisfying the usual pentagon axiom, and there is a unique object, up to isomorphism, $\one$, such that the functors 
\begin{equation}X\mapsto \one\ot X\text{ and }X\mapsto X\ot \one\end{equation} are both isomorphic to the identity functor.
\item The number of isomorphism classes of simple objects in $\C$ is finite.
\item The tensor unit $\one$ is a simple object in $\C$. 
\item The category $\C$ is rigid. This means that every object $X$ has a right dual $X^*$ and a left dual $^*X$.
The right dual is defined uniquely up to an isomorphism by the condition that there are maps $ev_X:X^*\ot X\to \one$
and $coev_X:\one\to X\ot X^*$ satisfying some coherence conditions.
The left dual is defined similarly. The semisimplicity of a fusion category implies that left and right duals are isomorphic. 
\end{enumerate}
A fusion category is called \textit{braided} if it is equipped with a natural isomorphism 
\begin{equation}\sigma_{X,Y}:X\ot Y\to Y\ot X\end{equation} for every two objects $X,Y\in\C$ such that for every $X,Y,Z\in \C$ the morphism 
\begin{equation}X\ot Y\ot Z\stackrel{\sigma_{X,Y}\ot 1_Z}{\to}Y\ot X\ot Z\stackrel{1_Y\ot\sigma_{X,Z}}{\to} Y\ot Z\ot X\end{equation}
is equal to the morphism 
\begin{equation}X\ot Y\ot Z\stackrel{\sigma_{X,Y\ot Z}}{\to} Y\ot Z\ot X \end{equation}
and the morphism 
\begin{equation}X\ot Y\ot Z\stackrel{1_X\ot \sigma_{Y,Z}}{\to}X\ot Z\ot Y\stackrel{\sigma_{X,Z}\ot 1_Y}{\to} Z\ot X\ot Y\end{equation} 
is equal to the morphism
\begin{equation} 
X\ot Y\ot Z\stackrel{\sigma_{X\ot Y,Z}}{\to}Z\ot X\ot Y \end{equation}
(to ease notations, we do not write here the associativity constraints). 
Notice that we do not assume that $\sigma_{X,Y}\sigma_{Y,X} = 1_{Y\ot X}$. A category satisfying this extra assumption is called \textit{symmetric}.
\end{definition}
One important example of a braided fusion category is the Drinfeld center of $Vec_G$. 
The objects in this category are vector spaces that admit a $G$-action and a $G$-grading. 
The action and the grading should be compatible in the following sense: for $g,h\in G$ we have $g\cdot V_h\subseteq V_{ghg^{-1}}$.
This category is braided. The braiding is given by the following formula:
\begin{equation}V\ot W\to W\ot V\end{equation}
$$v\ot w\mapsto g\cdot w\ot v \text{ for } v\in V_g.$$
This is an example of a braided monoidal category that is also \textit{modular}.

\subsection{Algebras, coalgebras, and Hopf algebras inside monoidal categories} An associative unital algebra inside a monoidal category $\C$ is defined as an object $A$ of the category together with morphisms
$m: A\ot A\to A$ and $u:\one\to A$ satisfying the associativity relation \begin{equation}m(m\ot 1_A) = m(1_A\ot m)$$ and the unit axiom $$1_A=m(1_A\ot u)=m(u\ot 1_A).\end{equation} A co-associative counital coalgebra is defined similarly, by changing the domain and codomain of all the relevant morphisms. 
A \emph{Hopf algebra} $R$ inside a braided monoidal category $\C$ is an object $R$ of $\C$ equipped with the following maps 
$$m:R\ot R\to R$$
$$u:\one\to R$$
\begin{equation}\Delta:R\to R\ot R\end{equation}
$$\epsilon:R\to\one\text{ and}$$
$$S:R\to R$$ such that the following conditions hold:
\begin{enumerate}
\item $(R,m,u)$ is an associative unital algebra.
\item $(R,\Delta,\epsilon)$ is a coassociative counital coalgebra.
\item $\Delta$ and $\epsilon$ are algebra maps. This means that the diagrams \\
\begin{equation}\xymatrix{
	& R\ot R\ot R\ot R\ar[rr]^{Id_R\ot \sigma_{R,R}\ot Id_R} & & R\ot R\ot R\ot R\ar[rd]^{m\ot m} & \\
	R\ot R\ar[ru]^{\Delta\ot\Delta}\ar[rrd]^{m} & & & & R\ot R \\
	& & R\ar[rru]^{\Delta} & & 
} \end{equation}

\begin{center}and   \\
\begin{equation}\xymatrix{R\ot R\ar[d]^{\epsilon\ot\epsilon} \ar[r]^{m} & R\ar[d]^{\epsilon} \\
	\one \ot \one\ar[r]^{m_{\one}} & \one
	}\end{equation} \end{center}
are commutative.
\item The map $S$ is an antipode. This means that the two compositions \begin{equation}m(S\ot Id_R)\Delta:R\to R \text{ and } m(Id_R\ot S)\Delta:R\to R\end{equation} are equal to $u\epsilon$.
\end{enumerate}
Notice that algebras and coalgebras can be defined in any monoidal category, whereas the definition of a Hopf algebra requires the braiding in the category. If $R$ is a Hopf algebra inside a braided fusion category $\C$, then the dual object $R^*$ is again a Hopf algebra, where the multiplication is given by $\Delta^*:R^*\to (R\ot R)^*\cong R^*\ot R^*$ and the other structure maps are defined similarly as the dual of the structure maps of $R$, see Section 5 of \cite{AS00} and the preliminaries in \cite{Zhang}. 

\begin{remark}
Some of the Hopf algebras in this paper will actually be graded Hopf algebras in the bigger category $\text{Ind}(\C)$, which contains infinite direct limits of diagrams in $\C$. To make it clear that a certain algebra is already contained in $\C$ we will say it is \textit{finite dimensional}. This is consistent with the notion of finite dimensionality when the category is the category of Yetter-Drinfeld modules over some finite dimensional Hopf algebra.
\end{remark}

We say that an associative algebra $A$ inside a fusion category is \emph{connected} if $A/J(A)\cong \one$, the trivial algebra in $\C$. Here $J(A)$ is the Jacobson radical of $A$, and is defined as the biggest nilpotent ideal in $A$. This definition makes sense in a general fusion category, and not only for finite dimensional algebras over a field. Indeed, an ideal $I$ of $A$ is a subobject of $A$ for which the image of the restriction of the multiplication maps \begin{equation}I\ot A\to A\text{ and }A\ot I\to A\end{equation} is contained in $I$. Nilpotency of the ideal means that for a big enough $N$, the multiplication map \begin{equation}I^{\ot N} = I\ot I\ot I\ot \cdots\ot I\to A\end{equation} is the zero map. 
In a similar way, we define a coalgebra $C$ to be coconnected, if its dual algebra $C^*$ is connected. This is equivalent to the coradical of $C$, which is the largest cosemisimple subcoalgebra of $C$, being isomorphic to $\one$.
\begin{definition} \label{def:ccc}
A Hopf algebra is called (co)connected if it is (co)connected as a (co)algebra. 
A Hopf algebra is called connected coconnected (or ccc) if it is both connected and coconnected.
\end{definition}
Among the ccc Hopf algebras the Nichols algebras play a prominent role (see \cite{AS02}). We recall now their definition.

\begin{definition}\label{def:Nichols}(see Subsection 5.7. in \cite{BB12})
For a given object $V\in \C$ the Nichols algebra $\B(V)$ is the unique Hopf algebra in $\text{Ind}(\C)$ that satisfies the following conditions:
\begin{enumerate}
\item The Hopf algebra $\B(V)$ is graded by the non-negative integers as a Hopf algebra.
\item The zeroth component of the grading satisfies $\B(V)_0=\one$. 
\item The first component of the grading satisfies $\B(V)_1=V$, and $\B(V)$ is generated by $V$.
\item The subobject of primitive elements of $\B(V)$ is $V$. This subobject is defined for any Hopf algebra as $\Pp(R)=\Ker(\Delta - u\ot 1 - 1\ot u)$.
\end{enumerate}
\end{definition} 
\begin{remark} 
One of the fundamental and very difficult questions in the study of Nichols algebras is to determine for which objects $V\in\C$ the Nichols algebra is finite dimensional. 
\end{remark}
The definition above gives us a concrete way to construct the Nichols algebra. 
Since $\B(V)$ is generated by $V$, and the elements of $V$ are primitive in $\B(V)$ we have a surjective Hopf algebra map $\pi:T(V)\to \B(V)$. The fact that $\B(V)$ is graded and the elements of $V$ are of degree 1 implies that the map $\pi$ is a graded map. 
The Hopf algebra $\B(V)$ can then be constructed from $T(V)$ in the following way: 
we divide $T(V)$ first by the Hopf ideal $I_1$ of $T(V)$ generated by the primitive elements of $T(V)$ in degrees $>1$. Then in the quotient graded Hopf algebra $T(V)/I_1$ we divide by the ideal $I_2$ generated by the primitive elements of degree $>1$ in this algebra, and continue inductively.
Notice that it might happen that by dividing out $I_1$ we get new primitive elements in $T(V)/I_1$. This is the reason we need to repeat this process.
See also the introduction in \cite{AG}. An equivalent definition is given by dividing out the kernel of the Woronowicz symmetriser, see Subsection 5.7. in \cite{BB12}. 

The Nichols algebra of $V$ and of $V^*$ are related in the following way:
Recall that for a graded Hopf algebra $R=\oplus_{n\geq 0} R_n$ in $\C$, in which all the homogeneous components are finite dimensional, the graded dual 
\begin{equation}S = \bigoplus_{n\geq 0} (R_n)^*\end{equation} is also a graded Hopf algebra. 
In case $R$ itself is finite dimensional, this is the same as the dual $R^*$. 
We claim the following (see also Lemma 5.5 in \cite{AS00} and Proposition 3.2.20 in \cite{AnGr} for the case the category is the category of Yetter-Drinfeld modules over a Hopf algebra):
\begin{lemma}\label{lem:nicholsdual}
The graded dual of $\B(V)$ is $\B(V^*)$. In particular $\B(V)$ is finite dimensional if and only if $\B(V^*)$ is finite dimensional.
\end{lemma}
\begin{proof}
Let $S$ be the graded dual of $\B(V)$. It holds that $S_0= (\B(V)_0)^* = \one$ and $S_1 = (\B(V)_1)^* = V^*$. 
We first claim that $\Pp(S) = S_1$. Notice first that $\Pp(S)$ is a graded subobject of $S$. Assume that $S_1=V^*\subsetneq \Pp(S)$. Let $n$ be the minimal integer $>1$ such that $W=\Pp(S)_n\neq 0$. Then 
\begin{equation}\langle W, V^{\cdot n}\rangle = \langle \Delta^{n-1}W,V\ot V\ot \cdots\ot V\rangle = \end{equation}
$$\sum_{i=0}^{n-1}\langle  \epsilon^i\ot W\ot \epsilon^{n-1-i},V^{\ot n}\rangle = 0.$$
We have used here the primitivity of $W$ to express $\Delta^{n-1}$ using $\epsilon$, the fact that the multiplication in $R$ is dual to the comultiplication in $S$ and the fact that $\epsilon(V) = 0$. But the above equation implies that $\langle W, \B(V)_n\rangle = 0$ since $\B(V)$ is generated by $V$. This implies that $W=0$, a contradiction.

We prove now that $S$ is generated by $V^*$. 
Assume that this is not the case. Let $n$ be the minimal integer $>1$ such that $(V^*)^{\cdot n}\subsetneq S_n$. 
Using semisimplicity, we can find a subobject $0\neq W\subseteq V^{\ot n}$ such that $\langle (V^*)^{\cdot n},W\rangle=0$. 
A dual argument to the argument above shows that $\Delta(W)\in S_0\ot S_n \oplus S_n\ot S_0$. This follows from the fact that by the minimality of $n$, it holds that $S_m = (V^*)^{\cdot m}$ for every $m<n$. 
This means that \begin{equation}\langle \Delta(W),S_m\ot S_{n-m}\rangle = \langle \Delta(W),(V^*)^{\cdot m}\ot (V^*)^{\cdot (n-m)}\rangle=\end{equation}
$$\langle W, (V^*)^{\cdot m+n-m}\rangle = \langle W,(V^*)^{\cdot n}\rangle = 0.$$
Since $S_0=\one$ this already implies that $W$ is primitive, which is a contradiction to $\Pp(\B(V))=V$. 
\end{proof}
 
\subsection{Actions of algebraic groups on affine varieties}\label{subsec:alggroups}
We recall the following framework and basic facts about actions of algebraic groups.
Let $\Ga$ be a reductive algebraic group acting algebraically on an affine variety $X$. This means that the map $\Ga\times X\to X$ is given by a polynomial map. 
The following holds (see Section 8.3 in \cite{Humphreys} and Lemma 3.3 in \cite{Newstead})
\begin{proposition}
	All the orbits of $\Ga$ in $X$ are locally closed. For a $\Ga$-orbit $\Ow$ in $X$, it holds that $\overline{\Ow}\backslash \Ow$ is the union of orbits of smaller dimension. In particular, an orbit of minimal dimension in $\overline{\Ow}$ is closed. 
\end{proposition}
\begin{proposition}\label{prop:sepinvs} If $W_1$ and $W_2$ are two closed disjoint $\Ga$-stable subsets of $X$, then there is an element $f\in K[X]^{\Ga}$ such that $f(W_1) = 1$ and $f(W_2)=0$. 
In other words- we can separate the subsets $W_1$ and $W_2$ by an invariant polynomial.
\end{proposition}

\end{section}

\begin{section}{Finite dimensional Hopf algebras in braided fusion categories}\label{sec:hopffusion}
Let $H$ be a Hopf algebra in a braided fusion category $\C$. We will use here of the coradical filtration of $H$. 
It will be easier to define this filtration using the radical filtration of the dual algebra (see also  Chapter IX of \cite{Sweedler} and \cite{AS98} for the case where $H$ is a Hopf algebra in $Vec_K$).
\begin{itemize}
\item We define the radical $J$ of $H^*$ as \begin{equation}\cap_{M} \Ker(H\to M\ot M^*),\end{equation} where the intersection is taken over all simple $H$-modules $M$ in $\C$, and where $H\to M\ot M^*$ is adjoint to the action map $H\ot M\to M$. Just like in the case of finite dimensional algebras over a field, the ideal $J$ is nilpotent. 
\item We define $H_n = (H^*/J^{n-1})^*\subseteq (H^*)^*\cong H$. 
\end{itemize}
An equivalent definition of $H_n$ is given inductively by \begin{equation}H_n =  \Ker (H\stackrel{\Delta}{\to} H\ot H\to (H/H_{n-1})\ot (H/H_0),\end{equation} where $H_0$ is the coradical of $H$, defined as the sum of all simple subcoalgebras of $H$.
%The coradical filtration of $H$ is defined as follows
%\begin{itemize}
%	\item $H_0\subseteq H$ is the coradical. It is the sum of all simple subcoalgebras of $H$. This is dual to the operation of taking the quotient of $H$ by its Jacobson radical. 
%	\item For every $n> 0$ we define $H_n = \Ker (H\stackrel{\Delta}{\to} H\ot H\to (H/H_{n-1})\ot (H/H_0)$. 
%	Equivalently, $H_n = \Ker\big(H\stackrel{\Delta^{n-1}}{\to}H^{\ot n}\to (H/H_0)^{\ot n}\big)$.
%\end{itemize}
%Assume now that the braided fusion category $\C$ is $Vec_K$. 
The fact that $J^a\cdot J^b = J^{a+b}$ translates to the dual property $\Delta(H_n)\subseteq \sum_{a+b\leq n} H_a\ot H_b$. 
We say that $H^*$ satisfies the Chevalley property (and that $H$ satisfies the dual Chevalley property) if the tensor product of semisimple $H^*$-modules (in $\C$) is again semisimple. By considering the action of the radical $J$, this is equivalent to saying that $\Delta(J)\subseteq J\ot H^* + H^*\ot J$.  

When $H$ satisfies the dual Chevalley property we get that $\Delta(J^n)\subseteq \sum_{a+b\leq n} J^a\ot J^b$ by using the Hopf axiom. 
Dualising, this implies that $H_a\cdot H_b\subseteq H_{a+b}$, and the coradical filtration thus gives us a Hopf filtration on $H$. 
%If $H$ satisfies the dual Chevalley property, i.e. the tensor product of semisimple $H$-comodules is again a semisimple $H$-comodule, then $H_0$ is also a subalgebra of $H$, and not only a subcoalgebra (see \cite{AGM}). This is equivalent to the dual property, stating that $In this case the filtration is a filtration of algebras as well, that is $H_i\cdot H_j \subseteq H_{i+j}$. 
This means that the associated graded object \begin{equation}\text{gr}H := \oplus H_n/H_{n-1}\end{equation} is a graded Hopf algebra. We will say that $H$ is \textit{coradically graded} if $H\cong \text{gr}H$ as Hopf algebras. The grading gives a split surjection $\pi:\text{gr}H\to H_0$ of Hopf algebras. 
Using the process of Bosonization (or Radford-Majid biproduct) one can also write this algebra in the form 
$\text{gr}H = R\# H_0$ where $R$ is a graded Hopf algebra in the category of Yetter-Drinfeld modules over $H_0$.
As a vector space \begin{equation}R=\{r\in \text{gr}H| (1\ot \pi)\Delta(r) = r\ot 1\in \text{gr}H\ot H_0\}\end{equation} and $R_0=\one$. 
See \cite{AS10} for the description of $R$ as a Hopf algebra in $_{H_0}^{H_0}\YD$. 

The following lemma is the first step in proving Theorem \ref{thm:main2}:
\begin{lemma}\label{lem:genHgenR}
Let $A$ be a finite dimensional semisimple Hopf algebra. 
The following conditions are equivalent:
\begin{enumerate}
\item Every finite dimensional Hopf algebra $H$ in which the coradical $H_0$ is isomorphic to $A$ is generated by the zeroth and first levels of its coradical filtration.
\item Every coradically graded finite dimensional Hopf algebra $H$ in which the coradical $H_0$ is isomorphic to $A$ is generated by the zeroth and first levels of its coradical filtration.
\item Every coradically graded finite dimensional Hopf algebra $R\in ^A_A\YD$ in which $R_0=\one$ is generated by its primitive elements.
\end{enumerate}
\end{lemma}
\begin{proof}
The first condition clearly implies the second one. On the other hand, if the second condition holds and $H$ is a Hopf algebra such that $H_0\cong A$, we can pass to the associated graded Hopf algebra $\text{gr}H$. Since this Hopf algebra is coradically graded the second condition implies that it is generated by its zeroth and first terms of the coradical filtration, and the same thus holds also for $H$ (see Lemma 2.2. in \cite{AS98}).

We next prove that the second and third conditions are equivalent. 
Indeed, if $H$ is coradically graded then the above discussion implies that $H\cong R\#H_0$ where $R$ is a coradically graded Hopf algebra in $_A^A\YD$. It then holds that $H$ is generated by its first and zeroth terms of its coradical filtration if and only if the same holds for $R$. But this is equivalent to $R$ being generated by its primitive elements.
\end{proof}
\begin{lemma}\label{lem:pgen} (see also Lemma 5.5 in \cite{AS00})
Assume that $R$ is a Hopf algebra inside a braided fusion category $\C$.
If $R$ is generated by $\Pp(R)$, and $R^*$ is generated by $\Pp(R^*)$, then $R$ is isomorphic to $\B(\Pp(R))$ (that is: $R$ is a Nichols algebra).
\end{lemma}
\begin{remark} This lemma also holds if one replaces the braided fusion category with a finite braided tensor category.
\end{remark}
\begin{proof}
Write $\Pp(R) = V$ and $\Pp(R^*) = W$.
Write $\pi:T(V)\to R$ for the resulting surjective Hopf-algebra map in $\text{Ind}(\C)$.
Let $p:T(V)\to\B(V)$ be the canonical surjection. 
We will show that $\pi$ splits via $p$. 

Assume that $U\subseteq T(V)$ is a primitive subobject of degree $n>1$ (that is: $U\subseteq \Pp(T(V))$). We will show that $\pi(U)=0$.
The primitivity of $W$ implies that 
\begin{equation}\langle W,V^{\cdot n}\rangle = \langle \sum_{i=0}^{n-1}\epsilon^i\ot W\ot \epsilon^{n-1-i},V^{\ot n}\rangle = 0.\end{equation}
So in particular $\langle W,\pi(U)\rangle = 0$. Using now the primitivity of $U$ (which also implies the primitivity of $\pi(U)$, since $\pi$ is a Hopf-algebra morphism) we get 
\begin{equation}\langle W^{\cdot m},\pi(U)\rangle = \langle W^{\ot m},\sum_{i=0}^{m-1}1^i\ot \pi(U)\ot 1^{m-i-1}\rangle=0\end{equation}
for every $m$. But this implies that $\pi(U)$ is perpendicular to the subalgebra of $R^*$ generated by $W$, which is $R^*$ itself.
This means that $\pi(U)=0$, so $U\subseteq \Ker(\pi)$.  

This implies that the ideal $I_1\subseteq T(V)$ generated by primitive elements of degree $>1$ is contained in $\Ker(\pi)$.
We thus get a surjective Hopf algebra map $\pi_1:T(V)/I_1\to R$. Denote by $I_2$ the ideal of $T(V)/I_1$ generated by primitive elements of degree $>1$ in this algebra. By the same argument, $I_2\subseteq \Ker(\pi_1)$. We define now inductively ideals $I_n$ and Hopf algebra surjections $\pi_n:T(V)/I_n\to R$ such that $I_n\subseteq T(V)/I_{n-1}$ is the ideal generated by all primitive elements of degree $>1$. This is the same as the chain of ideals which appears after Definition \ref{def:Nichols}.
The union of the inverse images of the ideals $I_n$ inside $T(V)$ is exactly the kernel of the surjection $p:T(V)\to \B(V)$. We thus get a surjective map 
\begin{equation}\wb{\pi}:\B(V)\to R\end{equation} which is injective on $V$. 
This map must be injective as well, due to the following reason: assume that $n$ is the minimal number such that $A:=\bigoplus_{i\leq n} \B(V)\cap \Ker(\wb{\pi})\neq 0$. Then by the fact that $\Ker(\wb{\pi})$ is a Hopf ideal and by considering the grading we get $\Delta(A)\subseteq A\ot \B(V) + \B(V)\ot A$ and therefore $\Delta^{+}(A)\in A\ot \B(V) + \B(V)\ot A$ as well, where $\Delta^{+}$ is defined to be zero on $\one$ and $\Delta-1\ot\text{Id}-\text{Id}\ot 1$ on $\Ker(\epsilon)$. It holds that \begin{equation}\Delta^{+}(A)\subseteq \bigoplus_{i+j\leq n, (i,j)\neq (0,n),(n,0)}\B(V)_i\ot \B(V)_j.\end{equation}
By the minimality of $n$, $\pi$ is injective on $\bigoplus_{i<n}\B(V)_i$, and so we get that $\Delta^+(A)=0$, but this contradicts the fact that all the primitive elements of $\B(V)$ are concentrated in degree 1. 
\end{proof}

\begin{lemma}\label{lem:pairing}
Let $R$ be a ccc Hopf algebra. 
The restriction of the pairing $R\ot R^*\to \one$ to $\Pp(R)\ot \Pp(R^*)$ is non-degenerate if and only if $R\cong \B(\Pp(R))$.
\end{lemma}
\begin{proof}
One direction follows from the fact that the graded dual of $\B(V)$ is $\B(V^*)$, see Lemma \ref{lem:nicholsdual}. 
Assume, on the other hand, that the pairing is non-degenerate.
By the previous lemma, it will be enough to prove that $R$ is generated by $V=\Pp(R)$.  
By a dual argument, $R^*$ is  generated by $W=\Pp(R^*)$, and we can use the Lemma \ref{lem:pgen} to finish the proof.

Write $J=\Ker(\epsilon)\subseteq R$. Consider the surjective morphism $R\to R/J^2$. By taking duals we get an injective morphism $(R/J^2)^*\to R^*$. 
We claim that $(R/J^2)^*=\Im(\epsilon)\oplus W$. 
The inclusion $\Im(\epsilon)\oplus W\subseteq (R/J^2)^*$ follows from the fact that $\epsilon(J^2)=0$ and that $\langle W,J\cdot J\rangle = \langle \Delta(W),J\ot J\rangle \subseteq
\langle W\ot \Im(\epsilon) + \Im(\epsilon)\ot W,J\ot J\rangle = 0$. 
In the other direction, the fact that $\epsilon(u)=1$ enables us to write $(R/J^2)^*\cong \Im(\epsilon)\oplus Q$ where $Q(\Im(u))=0$. We claim that $Q$ is contained in $W$. For this, let us write $R= \Im(u)\oplus J$. We will show that the morphism \begin{equation}\Delta- \text{Id}\ot\epsilon - \epsilon\ot\text{Id}:Q\to R^*\ot R^*\end{equation} vanishes (where we consider here $\epsilon$ as a morphism $\one\to R^*$). We will do so by showing that the evaluation of the image of this map on $R\ot R$ is zero. 

We write \begin{equation}R\ot R = \Im(u)\ot \Im(u) \oplus J\ot\Im(u)\oplus \Im(u)\ot J\oplus J\ot J.\end{equation} Since $\Im(u)\cdot\Im(u)=\Im(u)$, it holds that $\langle \Delta(Q),\Im(u)\ot\Im(u)\rangle = \langle Q,\Im(u)\rangle=0$. Similarly. we can show that the pairings of $Q\ot \Im(\epsilon)$ and $\Im(\epsilon)\ot Q$ with $\Im(u)\ot \Im(u)$ vanish. For $J\ot J$ we have that $\langle \Delta(Q),J\ot J\rangle = \langle Q,J\cdot J\rangle=0$ because $Q\subseteq (R/J^2)^*$ It also holds that $\langle (\text{Id}\ot\epsilon)(Q),J\ot J\rangle=\langle (\epsilon\ot\text{Id})(Q),J\ot J\rangle=0$, so the pairing of the image of the above morphism with that summand vanishes as well. 
For $\Im(u)\ot J$ we have $\langle \Delta(Q),\Im(u)\ot J\rangle = \langle Q,J\rangle$, 
$\langle (\epsilon\ot\text{Id})(Q),\Im(u)\ot J\rangle = \epsilon(\Im(u))\langle Q,J\rangle = \langle Q,J\rangle$, and   
$\langle (\text{Id}\ot\epsilon)(Q),\Im(u)\ot J\rangle = \langle Q,\Im(u)\rangle\langle \epsilon,J\rangle= 0$. The case for $J\ot \Im(u)$ is similar. 

So we know that $(R/J^2)^*=\Im(\epsilon)\oplus W$. Assume that the pairing $V\ot W\to \one$ is non-degenerate. Consider the image $\wb{V}$ of $V$ in $J/J^2$. We claim that $\wb{V}=J/J^2$. 
For this, we use the fact that $R/J^2 = \Im(u)\oplus J/J^2$ and $(R/J^2)^* = \Im(\epsilon)\oplus W$. If $\wb{V}$ is a proper sub-object of $J/J^2$, then by semisimplicity there is a proper subobject $A\subseteq W$ such that $\langle A,\wb{V}\rangle = \langle A,V\rangle=0$. But this contradicts the fact that the pairing $V\ot W\to \one$ is non-degenerate. 

Finally, there is a version of Nakayama's Lemma that holds here. The fact that the image of $V$ spans $J/J^2$ implies that $V$ generates $R$. Indeed, we can prove by induction that $V$ generates $R/J^n$ for every $n\geq 2$. If $V$ generates $R/J^n$ then in particular it generates $J^{n-1}/J^n$. Since the multiplication induces a surjective morphism $J^{n-1}/J^n\ot J/J^2\to J^n/J^{n+1}$, we see that $V$ generates $R/J^{n+1}$ as well. Since $J$ is a nilpotent ideal we are done.
\end{proof}

We are now ready to prove the equivalence of the third and fourth conditions of \ref{thm:main2}. 
\begin{proof}
Lemma \ref{lem:genHgenR} and Lemma \ref{lem:pgen} show that the third condition of Theorem \ref{thm:main2} is equivalent to the statement that all finite dimensional coradically graded Hopf algebras $R\in ^A_A\YD$ with $R_0=\one$ are Nichols algebras. 
If the fourth condition of Theorem \ref{thm:main2} holds and all the ccc Hopf algebras in $^A_A\YD$ are Nichols algebras, this is in particular true for coradically graded Hopf algebras with $R_0=\one$, and the fourth condition thus implies the third condition. 

If on the other hand the third condition holds and $R$ is a ccc Hopf algebra in $^A_A\YD$ then the Hopf algebra $\text{gr}R$ arising from the coradical filtration of $R$ is a Nichols algebra. This implies that $R$ is generated by its primitive elements. Similarly, the dual $R^*$ is also generated by its primitive elements, and by Lemma \ref{lem:pgen} $R$ is a Nichols algebra. 
\end{proof}
\end{section}

\begin{section}{Braided fusion categories by linear algebra}\label{sec:catslinear}
Let $\C$ be a braided fusion category.
Write $X_1,\ldots X_s$ for a set of representatives of the isomorphism classes of simple objects of $\C$. We would like to describe all the data encoded in the structure of $\C$ as a braided fusion category using vector spaces and linear maps.
We begin with the hom-spaces, which are very easy to describe. Indeed, it holds that 
\begin{equation}\label{eq:simples}\Hom_{\C}(X_i,X_j) = \begin{cases} 0 & \text{ if } i\neq j \\
									K\cdot Id_{X_i} & \text{ if } i=j \end{cases}
								\end{equation}
since $\C$ is a semisimple category, and the $X_i$ are non-isomorphic simple objects. Every object of $\C$ is isomorphic to a direct sum of simple objects.
Instead of writing an object of $\C$ as $\bigoplus_i X_i^{\oplus a_i}$ we will use the isomorphic object \begin{equation}\bigoplus_i U_i\ot X_i\end{equation} where $U_i$ are plain vector spaces. Notice that the fact that $\C$ is a $K$-linear category means that taking tensor products of objects in $\C$ with vector spaces makes sense.

The hom-spaces in $\C$ are then given by 
\begin{equation}\label{eq:homs} \Hom_{\C}(\bigoplus _i V_i\ot X_i,\bigoplus_j U_j\ot X_j) = \bigoplus_i \Hom_K(V_i,U_i)\end{equation}
We describe next the tensor product and the associativity constraints.
Assume that \begin{equation}[X_i]\cdot [X_j] = \sum_k N_{i,j}^k [X_k]\end{equation} in the Grothendieck ring of $\C$. For every three indices $i,j,k\in \{1,\ldots s\}$ fix a vector space $V_{i,j}^k$ of dimension $N_{i,j}^k$. We can then write
\begin{equation}X_i\ot X_j \cong  \bigoplus_k V_{i,j}^k\ot  X_k.\end{equation} Notice that 
\begin{equation}\big(\bigoplus_i U_i\ot X_i\big)\ot \big(\bigoplus_j W_j\ot X_j\big)\cong 
\bigoplus_{i,j,k} U_i\ot W_j\ot V_{i,j}^k\ot X_k.\end{equation}
For $i,j,k\in\{1,\ldots s\}$ we then have
\begin{equation}(X_i\ot X_j)\ot X_k\cong \bigoplus_a V_{i,j}^a\ot X_a\ot X_k\cong \bigoplus_{a,b} V_{i,j}^a\ot V_{a,k}^b \ot X_b\end{equation} while on the other hand
\begin{equation} X_i\ot (X_j\ot X_k)\cong \bigoplus_c X_i\ot V_{j,k}^c\ot X_c\cong \bigoplus_{c,b} V_{i,c}^b\ot V_{j,k}^c \ot X_b.\end{equation}
The associativity constraints \begin{equation}\alpha_{i,j,k}:(X_i\ot X_j)\ot X_k\to X_i\ot (X_j\ot X_k)\end{equation} are then given by a family of linear maps 
\begin{equation}\alpha_{i,j,k}^{a,b,c}: V_{i,j}^a\ot V_{a,k}^b\to V_{i,c}^b\ot V_{j,k}^c\end{equation}
that combine to give linear isomorphisms \begin{equation}\bigoplus_a V_{i,j}^a\ot V_{a,k}^b\stackrel{\cong}{\to} \bigoplus_c V_{i,c}^b\ot V_{j,k}^c\end{equation} for every $i,j,k,b$. 
The Pentagon axiom then translates into a list of axioms that says that 
certain sums of compositions of the linear maps $\alpha_{i,j,k}^{a,b,c}$ are equal. More precisely, for every $i,j,k,l\in \{1,\ldots s\}$ writing the pentagon diagram for the tensor product of $X_i\ot X_j\ot X_k\ot X_l$ gives us that for every $a,b,c\in\{1,\ldots s\}$ the composition

$$V_{i,j}^a\ot V_{a,k}^b\ot V_{b,l}^c\stackrel{\alpha_{i,j,k}^{a,b,d}}{\longrightarrow} 
\bigoplus_d V_{i,d}^b\ot V_{j,k}^d\ot V_{b,l}^c\stackrel{\alpha_{i,d,l}^{b,c,e}}{\longrightarrow} $$
\begin{equation}\bigoplus_{d,e}V_{i,e}^c\ot V_{j,k}^d\ot  V_{d,l}^e\stackrel{\alpha_{j,k,l}^{d,e,f}}{\longrightarrow} 
\bigoplus_{e,f} V_{i,e}^c \ot V_{j,f}^c\ot V_{k,l}^f\end{equation} is equal to the composition
\begin{equation} V_{i,j}^a\ot V_{a,k}^b\ot V_{b,l}^c\stackrel{\alpha_{a,k,l}^{b,c,f}}{\longrightarrow}
\bigoplus_f V_{i,j}^a\ot V_{a,f}^c\ot V_{k,l}^f\stackrel{\alpha_{i,j,f}^{a,c,e}}{\longrightarrow}
\bigoplus_{f,e} V_{i,e}^c\ot V_{j,f}^t\ot V_{k,l}^f.\end{equation}

Assuming that $X_1=\one$ is the tensor unit, the unit axioms for the monoidal category $\C$ can be translated as saying that $V_{1,i}^j$ and $V_{i,1}^j$ are zero if $i\neq j$, are one dimensional in case $i=j$, and that there are distinguished bases $l_i\in V_{1,i}^i$ and $r_i\in V_{i,1}^i$ such that for every $i,j,k\in\{1,\ldots s\}$ it holds that 
\begin{equation}\alpha_{i,1,j}^{i,k,j}:V_{i,1}^i\ot V_{i,j}^k\to V_{i,j}^k\ot V_{1,j}^j\end{equation} sends $r_i\ot v$ to $v\ot l_i$ for $v\in V_{i,j}^k$. 

The rigidity of the category can be described in this language in the following way:
for every $i\in \{1,\ldots s\}$ there is a unique $\bar{i}\in \{1,\ldots s\}$ such that $V_{i,\bar{i}}^1$ and $V_{\bar{i},i}^1$ are one-dimensional, and $V_{i,j}^1=0$ for $j\neq \bar{i}$.
The evaluation \begin{equation}ev_i:X_{\bar{i}}\ot X_i\to X_1\end{equation} is then given by a linear isomorphism which we denote by the same symbol $ev_i:V_{\bar{i},i}\to K$ and the coevaluation $X_1\to X_i\ot X_{\bar{i}}$ is then given by a linear isomorphism $coev_i:K\to V_{i,\bar{i}}$. The rigidity axioms translate again to equality between compositions of linear maps. 
The equality between the composition \begin{equation}X_i\to (X_i\ot X_{\bar{i}})\ot X_i\to X_i\ot(X_{\bar{i}}\ot X_i)\to X_i\end{equation} and $Id_{X_i}$ translates to the equality \begin{equation}K\to V_{i,\bar{i}}^1\ot V_{1,i}^i\to V_{i,1}^i\ot V_{\bar{i},i}^1\to K=K\stackrel{Id_K}{\to} K\end{equation} where the first map sends 1 to $coev_i\ot l_i$, the second map is $\alpha_{i,\bar{i},i}^{1,1,i}$ and the third map sends $r_i\ot v$ to $ev_i(v)\in K$. 
 
The braided structure is given by maps $\sigma_{X_i,X_j}:X_i\ot X_j\to X_j\ot X_i$. This is the same as a collection of linear isomorphisms
\begin{equation}\sigma_{i,j}^k:V_{i,j}^k\to V_{j,i}^k,\end{equation}
which should satisfy the axioms arising from the braid relations.

The introduction of the vector spaces $V_{i,j}^k$ here and the linear maps $\alpha_{i,j,k}^{a,b,c}$ and $\sigma_{i,j}^k$
can be seen as a way to ``introduce coordinates'' on the category $\C$. We use here the fact that as an abelian category, $\C$ is very simple to understand on the level of objects and morphisms. The additional braided monoidal structure is described using linear algebra. This will be used later on in the construction of the variety $X_B$.
This should be seen as more of an auxiliary result, and will not play a dominant role in the sequel.
\end{section}

\begin{section}{The variety $X_B$ and the action of the group}\label{sec:variety}
Usually, when one speaks of ``a Hopf algebra $R$ inside the braided fusion category $\C$'' it is understood that $R$ is an object of $\C$ that is equipped with structure maps which are not written explicitly.
We will take here a different approach. We will fix an object $B$ inside our braided fusion category $\C$, and ask what are \textit{all the possible} Hopf algebra structures one can give on that object. 

A Hopf algebra is given by morphisms $$m: B\ot B\to B,$$ $$u:K\to B,$$ \begin{equation}\Delta:B\to B\ot B\end{equation} $$\epsilon:B\to K \text{ and }$$ $$S:B\to B$$ which satisfy some axioms.
We can thus think of a Hopf algebra as a point in the affine space 
$$\A^N = \Hom_{\C}(B\ot B,B)\oplus\Hom_{\C}(K,B)\oplus\Hom_{\C}(B,B\ot B)$$ \begin{equation} \oplus\Hom_{\C}(B,K)\oplus\Hom_{\C}(B,B).\end{equation}
We will write a point in this space as $(m,u,\Delta,\epsilon,S)$. 
Notice that not all points in this affine space will define Hopf algebra structure, and not all Hopf algebras will be ccc.
We write \begin{equation}X_B\subseteq \A^N\end{equation} for the subset of all points $(m,u,\Delta,\epsilon,S)$ which define a ccc Hopf algebra structure on $B$.
For $t=(m,u,\Delta,\epsilon,S)\in X_B$ we write $(B,t)$ for the Hopf algebra $B$ with structure given by $t$. When we will say ``$R$ is a Hopf algebra in $X_B$'' we will mean that $R = (B,t)$ and $t\in X_B$. We will also write $R\in X_B$.  

Write $\Ga_B=\Aut_{\C}(B)$. When no confusion will arise we will write $\Ga = \Ga_B$. The group $\Gamma$ acts on the different direct summands in $\A^n$ by conjugation. 
The action of $\Gamma$ on $B\ot B$ is the diagonal one, and on $K$ is the trivial one. 
This induces a linear action of $\Ga$ on $\A^N$. We claim the following:
\begin{lemma}
	The action of $\Ga$ on $\A^N$ stabilizes the subset $X_B$. 
	Two points $(m_1,u_1,\Delta_1,\epsilon_1,S_1)$ and $(m_2,u_2,\Delta_2,\epsilon_2,S_2)$ in $X_B$ define isomorphic Hopf algebras 
	if and only if they lie in the same $\Ga$-orbit. The stabilizer of $(m,u,\Delta,\epsilon,S)$ in $\Ga$ can be identified with the group of automorphisms
	of the Hopf algebra that this tuple defines. 
\end{lemma}
\begin{proof}
	Write $t_i = (m_i,u_i,\Delta_i,\epsilon_i,S_i)$ for $i=1,2$. 
	The Hopf algebra axioms can be phrased as equalities between certain linear maps. Associativity of the multiplication, for example, is the equality 
	\begin{equation}m(1_B\ot m)\alpha_{B,B,B} = m(m\ot 1_B)\end{equation} as morphisms in $\Hom_{\C}((B\ot B)\ot B,B)$. If $\gamma:B\to B$ is an automorphism in $\C$, and if 
	$\gamma(t_1)=(t_2)$ then we have that
	$$m_2(1_B\ot m_2)\alpha_{B,B,B} = \gamma m_1 (\gamma^{-1}\ot \gamma^{-1})(1_B\ot \gamma m_1 (\gamma^{-1}\ot\gamma^{-1}))\alpha_{B,B,B} = $$
	\begin{equation} \gamma m_1(1_B\ot m_1)\alpha_{B,B,B} (\gamma^{-1}\ot \gamma^{-1}\ot \gamma^{-1})\end{equation} where we used the naturality of $\alpha$ and the definition of the action of $\gamma$. On the other hand
	a similar calculation gives us 
	\begin{equation}m_2(m_2\ot 1_B) = \gamma m_1(m_1\ot 1_B)(\gamma^{-1}\ot\gamma^{-1}\ot\gamma^{-1}).\end{equation}
	This shows that if $\gamma(t_1) = t_2$ then $m_1$ is associative if and only if $m_2$ is associative. 
	Similarly, all the other Hopf algebra axioms are valid for $t_1$ if and only if they are valid for $t_2$.
	So if $\gamma(t_1) = t_2$ then $t_1$ defines a Hopf algebra if and only if $t_2$ does.
	Moreover, $t_1$ will define a ccc Hopf algebra if and only if $\Ker(\epsilon_1)$ is a nilpotent ideal in $B$ with respect to the multiplication $m_1$ 
	and $\Ker(u_1^*)$ is a nilpotent ideal in $B^*$ with respect to the multiplication $\Delta_1^*$. 
	For the same reason as above, this happens if and only if $\Ker(\epsilon_2)$ is a nilpotent ideal with respect to the multiplication $m_2$
	and $\Ker(u_2^*)$ is a nilpotent ideal in $B^*$ with respect to the multiplication $\Delta_2^*$. 
	In other words, $\Gamma$ stabilizes the subset $X_B$ of $\A^N$. 
	 	
	We will think of the equation $\gamma(t_1) = t_2$ as saying that $\gamma$ defines an isomorphism between $(B,t_1)$ and $(B,t_2)$. 
	Indeed, $\gamma m_1 (\gamma^{-1}\ot \gamma^{-1})=m_2$ can be rephrased as saying that the diagram 
	\begin{equation}\xymatrix{ B\ot B\ar[r]^{\gamma\ot \gamma}\ar[d]^{m_1} & B\ot B\ar[d]^{m_2} \\ B\ar[r]^{\gamma} & B}\end{equation} is commutative.
	Similar statements hold for $u,\Delta,\epsilon$ and $S$. This implies that if $\gamma(t_1)=t_2$ then $t_1$ and $t_2$ define isomorphic Hopf algebras.
	On the other hand, if $t_1$ and $t_2$ define isomorphic Hopf algebras on $B$, take an isomorphism $\gamma:(B,t_1)\to (B,t_2)$ between these Hopf algebras.
	Then by the same calculations as above we get that $\gamma(t_1)=t_2$. 
	So the orbits of $\Ga$ in $X_B$ are in one to one correspondence with isomorphism classes of ccc Hopf algebras which are isomorphic to $B$ as an object of $\C$.
	Finally, the equality $\gamma(t) = t$ for $t\in X_B$ just means that $\gamma:(B,t)\to (B,t)$ is an automorphism.
\end{proof}

The rest of this section will be devoted to prove the following claim:
\begin{lemma}
	The subset $X_B$ is an affine sub-variety of $\A^N$. The group $\Ga$ is isomorphic to a direct product of general linear groups,
	and the action of $\Ga$ on $\A^N$ is algebraic.
	\end{lemma}

\begin{proof} The proof of the lemma will be based on analyzing objects and morphisms in the category $\C$.
We begin by writing $B$ as \begin{equation}B=\bigoplus_i B_i\ot X_i\end{equation} where $X_i$ are representatives of the isomorphism classes of simple objects of $\C$ and $B_i$ are vector spaces.
By Equation \ref{eq:homs} this already gives us an isomorphism 
$\Ga\cong \prod_i \GL(B_i)$. 
We choose a basis $\{e_{i1},\ldots e_{id_i}\}$ for $B_i$. 

A Hopf algebra structure on $B$ will be given by maps $m:B\ot B\to B$, $u:\one\to B$, $\Delta:B\to B\ot B$, $\epsilon:B\to \one$ and $S:B\to B$. By writing tensor products using the spaces $V_{i,j}^k$ from Section \ref{sec:catslinear} we see that the morphism $m$ is given by 
\begin{equation}\bigoplus_{i,j} B_i\ot B_j\ot X_i\ot X_j\to \bigoplus_k B_k\ot X_k.\end{equation}
Rewriting the first object using the vector spaces $V_{i,j}^k$ gives us \begin{equation}\bigoplus_{i,j,k}B_i\ot B_j\ot V_{i,j}^k\ot X_k\to \bigoplus_k B_k\ot X_k.\end{equation}
This means that $m$ is equivalent to a collection of linear maps \begin{equation}m_{i,j}^k:B_i\ot B_j\ot V_{i,j}^k\to B_k.\end{equation} Similarly, the morphism $u$ is equivalent to a map \begin{equation}u_1:K\to B_1,\end{equation}
the morphism $\Delta$ is equivalent to a collection of maps 
\begin{equation}\Delta_k^{i,j}:B_k\to B_i\ot B_j\ot V_{i,j}^k,\end{equation}
the morphism $\epsilon$ is equivalent to a map 
\begin{equation}\epsilon_1:B_1\to K\end{equation} and the antipode $S$ is equivalent to a collection of linear maps 
\begin{equation}S_i:B_i\to B_i.\end{equation}

We rewrite now the affine variety $\A^N$ as 
$$\A^N = \bigoplus_{i,j,k} Hom_K(B_i\ot B_j\ot V_{i,j}^k, B_k)\oplus Hom_K(K,B_1) \oplus $$ \begin{equation}\bigoplus_{i,j,k} Hom_K(B_k,B_i\ot B_j\ot V_{i,j}^k)\oplus Hom_K(B_1,K)\oplus \bigoplus_i Hom_K(B_i,B_i).\end{equation}

This description of $\A^N$ shows us that the action of  $\Aut_{\C}(B) =\prod_i \GL(B_i)$ on it is algebraic. Indeed, it is simply given by pre- and post-composing of linear maps.

A choice of bases for $B_i$ and for $V_{i,j}^k$ for all $i,j,k$ will give us a basis for $\A^N$. This enables us to describe the structure we have at hand, the tuple $t=(m,u,\Delta,\epsilon,S)$, as a collection of numbers, the \textit{structure constants} of $t$. Indeed, using the bases for $B_i$ and $V_{i,j}^k$ we can describe the different structure maps as linear maps between vector spaces with given bases, and these are just given by matrices of scalars.

We explain now why the subset $X_B$ is in fact an affine variety.
The idea is to show that all the Hopf algebra axioms can be expressed using polynomial equations. We will also show that the property of being ccc can be described using polynomial equations.

We begin with proving this for the associativity. The proof for the other Hopf algebra axioms is similar. 
The associativity axioms says that $m(m\ot 1_B) = m(1_B\ot m)\alpha_{B,B,B}$. Writing this using the maps $\alpha_{i,j,k}^{a,b,c}:V_{i,j}^a\ot V_{a,k}^b\to V_{i,c}^b\ot V_{j,k}^c$ and $m_{i,j}^k$ we get that associativity is equivalent to the commutativity of the diagram 
\begin{equation}\xymatrix{ B_i\ot B_j\ot B_k\ot V_{i,j}^a\ot V_{a,k}^b\ar[rr]^{\sum_c\alpha_{i,j,k}^{a,b,c}}\ar[d]^{m_{i,j}^a} & & \bigoplus_c B_i\ot B_j\ot B_k\ot V_{i,c}^b\ot V_{j,k}^c\ar[d]^{m_{j,k}^c} \\
B_a\ot B_k\ot V_{a,k}^b\ar[d]^{m_{a,k}^b} & & \bigoplus_c B_i\ot B_c\ot V_{i,c}^b\ar[d]^{m_{i,c}^b} \\ 
B_b\ar[rr]^{=} & & B_b } \end{equation}
where we simplified the morphisms by writing $\alpha_{i,j,k}^{a,b,c}$ instead of $1_{B_i\ot B_j\ot B_k}\ot \alpha_{i,j,k}^{a,b,c}$ and similarly for the other morphisms. The linear map $m_{i,j}^a$ goes from  $B_i\ot B_j\ot V_{i,j}^a$ to $B_a$, but it is clear how to get a linear map from it as shown in the diagram.
If we write now the linear maps $m_{i,j}^k$ and $\alpha_{i,j,k}^{a,b,c}$ in terms of the bases of $B_i$ and $V_{i,j}^k$ we get a set of quadratic polynomials on the structure constants whose vanishing is equivalent to the associativity of $m$. 
To state this more precisely, let us write $B_i=span\{e^i_r\}_r$ and $V_{i,j}^a= span\{v_{i,j,r}^a\}_r$. The map $\alpha^{a,b,c}_{i,j,k}:V_{i,j}^a\ot V_{a,k}^b\to V^{i,c}_b\ot V_{j,k}^c$ can be written as \begin{equation} v_{i,j,r_1}^a\ot v_{a,k,r_2}^b\mapsto \sum_{s_1,s_2}\alpha^{a,b,c,r_1,r_2}_{i,j,k,s_1,s_2}v^b_{i,c,s_1}v^c_{j,k,s_2}\end{equation} and the multiplication map $m_{i,j}^a$ can be written as \begin{equation} e^i_{r_1}\ot e^j_{r_2}\ot v_{i,j,r_3}^a\mapsto \sum_{r_4}m_{i,j,r_4}^{a,r_1,r_2,r_3}e^a_{r_4}.\end{equation} We consider now the basis element $e^i_{r_1}\ot e^j_{r_2}\ot e^k_{r_3}\ot v^a_{i,j,r_4}\ot v^b_{a,k,r_5}$. The commutativity of the above diagram is then equivalent to the vanishing of the quadratic polynomials
\begin{equation}\sum_{r_6}m_{i,j,r_6}^{a,r_1,r_2,r_4}m_{a,k,t}^{b,r_6,r_5,r_3} - \sum_{s_1,s_2,s_3}\alpha_{i,j,k,s_1,s_2}^{a,b,c,r_4,r_5}m_{j,k,s_3}^{c,r_2,r_3,s_2}m_{i,c,t}^{b,r_1,s_3,s_1},\end{equation}
where $t\in\{1,\ldots, \dim(B_b)\}$. 
The scalars $\alpha_{i,j,k,s_1,s_2}^{a,b,c,r_1,r_2}$ depend only on the fusion category, and we can consider them as constants. We thus get quadratic polynomials on the set of variables $m_{i,j,r_4}^{a,r_1,r_2,r_3}$. 
For similar reasons the other Hopf algebra axioms can be written as well as polynomials in the structure constants. 

It is left to show that being ccc is a closed condition for Hopf algebras. For this, write $n=\sum_i \dim_K(B_i)$. We claim the following:
\begin{lemma} An ideal $J$ of $(B,t)$ is nilpotent if and only if $J^n=0$.  
\end{lemma}
\begin{proof}
	One direction is clear. For the other direction, if $J$ is nilpotent, the sequence of ideals $J\supsetneq J^2\supsetneq J^3\cdots  $ is 
	strictly monotonic decreasing until it stabilizes at zero. Thus, the sequence $d_a:=\dim \Hom_{\C}(\oplus_i X_i, J^a)$ satisfy 
	\begin{equation}n > d_{1} > d_{2}> \ldots \end{equation} and this sequence of numbers stabilizes at zero. This implies that if $J$ is nilpotent its $n$-th power must already be zero. 
\end{proof}

By definition, a Hopf algebra $(B,t)$ is ccc if it is connected and coconnected. We will show that being connected is a closed condition. The fact that coconnectedness is a closed condition follows from a dual argument. 
By definition of connectedness, $(B,t)$ is connected if the ideal $\Ker(\epsilon)$ is nilpotent. This is equivalent to $\Ker(\epsilon)^n=0$ by the above lemma. It holds that the map \begin{equation}P=1_B-u\epsilon:B\to B \end{equation}is a projection on $\Ker(\epsilon)$ (this holds in any Hopf algebra, and follows from the fact that $\epsilon\circ u=\text{Id}_{\one}$).
The nilpotency of $\Ker(\epsilon)$ is thus equivalent to the fact that the map $$B^{\ot n}\stackrel{P^{\ot n}}{\longrightarrow} B^{\ot n}\stackrel{m^{n-1}}{\longrightarrow} B$$ is the zero map.
But again, this can be written as a polynomial equation using the structure constants of $u$, $\epsilon$ and $m$. The subset $X_B$ is thus a closed subvariety of $\A^N$, and the action of $\Ga$ on it is algebraic. The theory of algebraic groups and geometric invariant theory can thus be applied in our setting (see the results in Subsection \ref{subsec:alggroups}).
\end{proof}
\begin{definition}\label{def:leadsto}
If $R_1$ and $R_2$ are two Hopf algebras in $X_B$, we say that $R_2$ is a specialization of $R_1$ and write $R_1\lt R_2$ if $\Ow_{R_2}\subseteq \wb{\Ow_{R_1}}$.
\end{definition}
\begin{remark}
We have used here a slightly heavy categorical language, in order to construct the variety $X_B$ in the most general way possible. 
If, for example, the category $\C$ is $\GGYD$ there is a way around this: we can fix only the dimension of $B$ as a vector space, and consider also the action and coaction of $KG$ as part of the structure of $B$, instead of something that is given a-priori, as we have done here. Describing the isomorphism type of $B$ as an object in $\C$ can then be done by declaring what the trace of the operations of the elements of $D(G)$, the Drinfeld double of $G$, on $B$, should be. The construction here relies heavily on the fact that the category we are working in is semisimple. Indeed, the semisimplicity gives us an easy classification of the object of the category and their automorphism groups. See also \cite{AA18} for the study of Nichols algebras in non-semisimple braided monoidal categories.
\end{remark}
\begin{remark}\label{rmk:gradings} The Hopf algebras in $\GGYD$ which one encounters in the study of pointed Hopf algebras are usually graded. 
We study here ccc Hopf algebras and not graded algebras for two reasons. Firstly, 
being connected and coconnected is a conditions which can be described by polynomial equations. 
If we consider instead the variety of all graded Hopf algebras we will get something which is too rigid, and we will not be able to see the specializations in the orbits of $X_B$. Secondly, all finite dimensional graded Hopf algebras with $R_0=\one$ are automatically ccc. Indeed, this follows from the fact that if $R=\bigoplus_{i\geq 0}R_i$ then the Jacobson radical is $\bigoplus_{i\geq 1}R_i$, and the quotient is isomorphic to $\one$. The same holds for the dual. 
\end{remark}

\begin{remark}\label{rmk:duality} The map $R\mapsto R^*$ gives us an isomorphism of varieties $X_B\cong X_{B^*}$ which commutes with the action of $\Aut_{\C}(B)\cong\Aut_{\C}(B^*)$. In particular, if $R_1,R_2\in X_B$ then $R_1\lt R_2$ if and only if $R_1^*\lt R_2^*$ in $X_{B^*}$.
\end{remark}

The following lemma will be useful for the proof of Theorem \ref{thm:main3}.
\begin{lemma}\label{lem:subobj}
Assume that $R\lt R'$. Then $\Pp(R)$ is isomorphic to a subobject of $\Pp(R')$.
\end{lemma}
\begin{proof}
The object of primitive elements in $R$ is the same as the kernel of the map $$T_R=\Delta - u\ot 1_R - 1_R\ot u : R\to R\ot R$$
We can write this map as the direct sum of maps $B_i\ot X_i\to \oplus_{j,k}B_j\ot B_k\ot V_{i,j}^k\ot X_i$. 
Such a map is thus equivalent to a collection of maps $(T_R)_i:B_i\to \bigoplus_{j,k}B_j\ot B_k\ot V_{i,j}^k$.
For every linear map $L$ and any natural number $m$ the condition that $rank(L)\leq m$ is a Zariski closed condition. 
This implies that for every $i$ it holds that \begin{equation}rank((T_{R'})_i)\leq rank((T_R)_i)\end{equation} and therefore
\begin{equation}\dim(\Ker(T_{R'})_i)\geq \dim(\Ker((T_R)_i).\end{equation}
This gives us the desired result.
\end{proof}

\begin{proof}[Proof of Theorem \ref{thm:main3}]
If $R\lt \B(V)$ then $\Pp(R)$ is isomorphic to a subobject of $\Pp(\B(V)) = V$. 
Similarly $R^*\lt \B(V)^*\cong \B(V^*)$ by Lemma \ref{lem:nicholsdual} and Remark \ref{rmk:duality} so $\Pp(R^*)$ is isomorphic to a subobject of $\Pp(\B(V^*))=V^*$. 
Since $V$ is simple, $V^*$ is simple as well, and it follows that $\Pp(R)= 0$ or $V$ and $\Pp(R^*)= 0$ or $V^*$.
The option $\Pp(R^*)= 0$ is not possible, since this would imply that $J/J^2=0$ where $J$ is the Jacobson radical of $R$, and $R$ must then be the trivial algebra.
In a similar way, $\Pp(R)=0$ is impossible. We are left with the situation where $\Pp(R) = V$ and $\Pp(R^*)=V^*$. 
But this already implies that $\Pp(R)\ot \Pp(R^*)\to \one$ is a non-degenerate pairing, which implies by Lemma \ref{lem:pairing} that $R$ is a Nichols algebra, as desired.
\end{proof}

\end{section}

\begin{section}{Filtrations of Hopf algebras}\label{sec:filtrations}
As we have seen in previous sections, understanding specializations is fundamental to understand ccc Hopf algebras. 
Using the Hilbert-Mumford criterion, we will show that if $\Ow_{R_1}\subseteq \wb{\Ow_{R_2}}$ and $\Ow_{R_1}$ is closed, the specialization $R_2\lt R_1$ follows from a filtration on $R_2$, in a way which we will describe now. Let $R=(B,t)$ be a ccc Hopf algebra in $\C$. 
\begin{definition}[see also \cite{Makhlouf}] A Hopf algebra filtration on $R$ is a chain of $\C$-subobjects of $R$ 
$$\cdots R_{-2}\subseteq R_{-1}\subseteq R_0\subseteq R_1\subseteq\cdots$$ such that the following properties hold:
$$R_i = R \text{ for } i>>0,$$
$$R_i=0 \text{ for } i<<0,$$
$$\forall i,j: m(R_i\ot R_j)\subseteq R_{i+j},$$
\begin{equation}\forall k: \Delta(R_k)\subseteq \sum_{i+j= k} R_i\ot R_j,\end{equation}
$$\epsilon(R_{-1})=0,$$
$$\text{Im}(u)\subseteq  R_0,$$
$$\forall i: S(R_i)\subseteq R_i.$$
We will denote the filtration by $(R_i)$. 
\end{definition}
\begin{lemma} Every Hopf algebra filtration $(R_i)$ of $R$ determines a $\Z$-graded Hopf algebra $\text{gr}R$ with the same underlying object $B$, defined as follows:
\begin{enumerate}
\item As an object of $\C$, $$R=\oplus_{i\in \Z} R_i/R_{i-1}.$$ 
\item The unit $\wb{u}\in \text{gr}R$ is the image of $u$ in $R_0/R_{-1}$.
\item The counit $\wb{\epsilon}:\text{gr}R\to K$ is given by $\oplus R_i/R_{i-1}\to R_0/R_{-1}\stackrel{\epsilon}{\to} K$ where $\epsilon$ here is the map which is induced from $\epsilon$, since $\epsilon(R_{-1})=0$.
\item The condition on the antipode $S$ implies that it defines a collection of induced maps $\wb{S}_i:R_i/R_{i-1}\to R_i/R_{i-1}$. The antipode $\wb{S}$ of $\text{gr}R$ is $\sum_i \wb{S}_i$.
\item The condition on $m$ implies that we have an induced map $\wb{m}_{i,j}:R_i/R_{i-1}\ot R_j/R_{j-1}\to R_{i+j}/R_{i+j-1}$ for every $i,j\in \Z$. We define $\wb{m}=\sum_{i,j} \wb{m}_{i,j}$.
\item The condition on $\Delta$ implies that we have an induced map $\wb{\Delta}_{i,j}:R_{i+j}/R_{i+j-1}\to R_i/R_{i-1}\ot R_j/R_{j-1}$ for every $i,j\in\Z$. We define $\wb{\Delta}=\sum_{i,j}\wb{\Delta}_{i,j}$. 
\end{enumerate}
\end{lemma}
\begin{proof}
We first claim that, as objects of $\C$, we can write $R$ as $R=\oplus_{i\in \Z} T_i$ where $R_i = \oplus_{j\leq i} T_j$ where $T_i$ are subobjects of $R$ in $\C$. This follows from the semisimplicity of $\C$ together with the conditions on the filtration $R_i$. Indeed, take $i<<0$ for which $R_i=0$. Define $T_j=0$ for all $j\leq i$. Then choose $T_{i+k}$ inductively to be a direct sum complement of $T_{i+k-1}$ in $R_{i+k}$. Notice that this implies that $T_j\cong R_j/R_{j-1}$ for every $j\in \Z$, and that $B\cong R\cong\oplus T_j\cong \oplus R_j/R_{j-1}\cong \text{gr}R$ as objects of $\C$.  

All the structure maps of $\text{gr}R$ are well defined due to the condition the structure maps of $R$ satisfy.
%In the next lemma we will prove that $\text{gr}R$ is indeed a ccc Hopf algebra. 
\end{proof}
For the next lemma, recall that $X_B\subseteq \A^N$. We can consider the orbit $\Ow_{\text{gr}R}$ inside $\A^N$. 
\begin{lemma}
Let $(R_i)$ be a Hopf algebra filtration of the ccc Hopf algebra $R$. Then 
$\Ow_{\text{gr}R}\subseteq \wb{\Ow_R}$. In particular, $\Ow_{\text{gr}R}$ is also contained in the closed subset $X_B$ of $\A^N$, and as a result $\text{gr}R$ is also a ccc Hopf algebra.
\end{lemma}
\begin{proof}
We use the $\C$-objects $T_i$ constructed in the previous lemma. We write all the structure maps of $R$ in terms of the direct sum decomposition $R=\oplus_i T_i$. The conditions on the filtration gives us 
$$m=\sum_{i+j\geq k}m_{i,j}^k, 
\Delta= \sum_{i+j\leq k} \Delta_k^{i,j},$$
\begin{equation}u = \sum_{i\leq 0} u^i, \quad
\epsilon = \sum_{i\geq 0}\epsilon_i\end{equation}
$$S = \sum_{i\geq j}S_i^j$$
 where 
$$m_{i,j}^k:T_i\ot T_j\to T_k,\quad \Delta_k^{i,j}: T_k\to T_i\ot T_j,$$
\begin{equation}u^i:K\to T_i, \quad \epsilon_i:T_i\to K,\text{ and }\end{equation} $$S_i^j:T_i\to T_j.$$ 
Using the identification $R\cong \oplus_i T_i\cong \oplus_i R_i/R_{i-1}\cong \text{gr}R$ we see that the multiplication in $\text{gr}R$ is given by $\sum_{i,j}m_{i,j}^{i+j}$, the comultiplication is given by $\sum_{i,j}\Delta_{i+j}^{i,j}$, the unit by $u=u_0$, the counit is $\epsilon^0$, and the antipode is $\sum_i S_i^i$. We thus see that in passing from $R$ to $\text{gr}R$ we ``deleted'' all the parts of the structure maps that are of positive degree and stayed only with maps of degree zero (maps of negative degree do not appear here at all).  Here the degree of a map $T_{i_1}\ot \cdots\ot T_{i_r}\to T_{j_1}\ot\cdots\ot T_{j_m}$ is $i_1+i_2+\cdots + i_r-j_1-j_2-\cdots-j_m$. 

We use this idea to prove that the orbit of $\text{gr}R$ in $\A^N$ is in the closure of the orbit of $R$. This will already imply that $\text{gr}R$ is a ccc Hopf algebra, because $\Ow_R\subseteq X_B$, and $X_B$ is closed in $\A^N$. To prove this, we will use a \textit{one-parameter subgroup} of $\Ga$. That is: a group homomorphism $\phi:\Gg_m\to \Ga= \Aut_{\C}(B)$. We define $\phi$ as follows
\begin{equation}\phi(\lambda) = \sum_{i\in \Z}\lambda^{-i}\text{Id}_{T_i}.\end{equation}
We claim that $\phi(\Gg_m)(R)$ contains $\text{gr}R$ in its closure.
Since $\phi(\Gg_m)$ is a subgroup of $\Ga$ this will be enough.
To prove this write \begin{equation}\phi(\lambda)(R) = (B,t_{\lambda}) =  (B,m_{\lambda},u_{\lambda},\Delta_{\lambda},\epsilon_{\lambda},S_{\lambda}).\end{equation} 
We get $$m_{\lambda} = \sum_{i+j\geq k } \lambda^{i+j-k}m_{i,j}^k,\quad
\Delta_{\lambda} = \sum_{i+j\leq k} \lambda^{k-i-j}\Delta_k^{i,j}$$
\begin{equation}u_{\lambda} = \sum_{i\leq 0} \lambda^{-i}u^i,\quad 
\epsilon_{\lambda} = \sum_{i\geq 0} \lambda^i\epsilon_i\end{equation}
$$S_{\lambda} = \sum_{i\geq j}\lambda^{i-j}S_i^j.$$
This description shows us that $\lim_{\lambda\to 0}\phi(\lambda)(R)$ exists, since in all the above expression $\lambda$ appears only with non-negative powers. Taking the limit $\lambda\to 0$ gives us $\lim_{\lambda \to 0}\phi(\lambda)(B,t)= (B,m',u',\Delta',\epsilon',S')$
where 
$$m' = \sum_{i,j}m_{i,j}^{i+j} = \wb{m}, \Delta' = \sum_{i,j}\Delta_{i+j}^{i,j}=\wb{\Delta}$$
\begin{equation}u' = u_0 = \wb{u}, \epsilon' = \epsilon_0 = \wb{\epsilon}\end{equation}
$$S' = \sum_i S_i^i = \wb{S}$$
This shows us that the limit point is exactly the structure constants of $\text{gr}R$. We are done.
\end{proof}
\begin{remark}
If $R\in X_B$ and $\phi:\Gg_m\to \Ga$ is any one-parameter subgroup for which $\lim_{\lambda\to 0}\phi(\lambda)(R)$ exists, 
then we get a filtration on $R$ by setting $$T_i= \Ker(\phi(\lambda)-\lambda^{-i})\subseteq R\text{ for a generic } \lambda$$ $$R_i = \oplus_{j\leq i} T_j.$$
The fact that $\lim_{\lambda\to 0} \phi(\lambda)(R)$ exists implies, by the same argument as above, that $(R_i)$ is a Hopf algebra filtration. We thus see that the isomorphism classes of ccc Hopf algebras which appear on the boundary of $R$ by the action of a 1-parameter subgroup are exactly the ccc Hopf algebra arising from $R$ by a Hopf algebra filtration.
Moreover, by the Hilbert-Mumford criterion if $\Ow_{R'}\subseteq \wb{\Ow_R}$ and $\Ow_{R'}$ is closed, then there exists a one-parameter subgroup $\phi:\Gg_m\to \Ga$ such that $R'=\lim_{\lambda\to 0}\phi(\lambda)(R)$ (see Theorem 1.4. in \cite{Kempf}). This implies that in order to understand specializations, and especially specializations to ccc Hopf algebras with closed orbits, we need to study Hopf algebra filtrations.
\end{remark}

We finish this section with two filtrations which are canonically associated to any ccc Hopf algebra: the radical and the coradical filtration. They are dual to one another, in a way which we shall explain below. As was explained in the introduction, the coradical filtration of a Hopf algebra (not necessarily a ccc one) is used in a fundamental way in the classification of non-semisimple Hopf algebras. The use of both filtrations together will play an important role in studying closure of orbits in this paper.

For the radical filtration, let $J=\Ker(\epsilon)$ be the Jacobson radical of $R$. We claim the following:
\begin{lemma}
The filtration $R_i = J^{-i}$ for $i<0$ and $R_i=R$ for $i\geq 0$ is a Hopf algebra filtration.
\end{lemma}
\begin{proof}
Since $J$ is a nilpotent ideal, the condition $R_i=0$ for $i<<0$ holds. It is clear that the condition $R_i=R$ for $i>>0$ holds as well (it holds, in fact, for $i=0$).
The fact that $J^i\subseteq J^j$ when $j\leq i$ implies that $R_i\cdot R_j\subseteq R_{i+j}$. 
The condition $u\in R_0$ is immediate, and the condition $\epsilon(R_{-1})=0$ follows from the fact that $J=\Ker(\epsilon)$. For the condition on $\Delta$, notice that $\Delta(J) \subseteq J\ot R + R\ot J$. For any $i,j\geq 0$ the braiding $\sigma:R\ot R\to R\ot R$ satisfies $\sigma(J^i\ot J^j) = J^j\ot J^i$. A direct calculation implies that $\Delta(J^k) \subseteq \sum_{i+j=k} J^i\ot J^j$, which is what we wanted to prove. 
\end{proof}
We write $R_{gra}$ for the graded Hopf algebra arising from $R$ via the radical filtration.

We recall here also the definition of the dual filtration, the coradical filtration, from Section \ref{sec:hopffusion}:
We define $R_i=0$ for $i<0$, $R_0 = \text{Im}(u)$ and \begin{equation}R_n = \Ker\Big(R\stackrel{\Delta^{n-1}}{\to}R^{\ot n}\to (R/R_0)^{\ot n}\Big)\text{ for } n>0.\end{equation}
Again, a direct verification, using the fact that the dual algebra is connected, reveals the fact that this is a filtration of Hopf algebras as well. We write $R_{grc}$ for the graded Hopf algebra associated to this filtration. We thus see that $R\lt R_{gra}$ and $R\lt R_{grc}$ for every ccc Hopf algebra $R$ in $X_B$. 
The two filtrations are dual to one another in the following sense: For $i\in \Z$ let $R_i$ be the $i$-th level of the coradical filtration, and let $X_i\subseteq R^*$ be defined as \begin{equation}X_i = \Ker(R^*\to R_{-i}^*).\end{equation}
Then it holds that $(X_i)$ is the radical filtration on $R^*$. 
This duality induces isomorphisms $(R_{gra})^*\cong (R^*)_{grc}$ and $(R_{grc})^*\cong (R^*)_{gra}$.  
 \end{section}

\begin{section}{A proof of Theorem \ref{thm:main1} and \ref{thm:main4}}\label{sec:proofthm1}
In this section we prove that a ccc Hopf algebra in $X_B$ has a closed orbit if and only if it is isomorphic to a Nichols algebra.
Let $R$ be a Hopf algebra in $X_B$. Recall the associated graded Hopf algebras $R_{gra}$ and $R_{grc}$ from Section \ref{sec:filtrations}.
We know that $R\lt R_{gra}$ and $R\lt R_{grc}$. In particular, if the orbit of $R$ is closed we get that $R\cong R_{grc}$ and $R\cong R_{gra}$. The following proposition shows that if $\Ow_R$ is closed then $R$ is a Nichols algebra. 
\begin{proposition}\label{prop:gragrc}
Assume that the Hopf algebras $R_{gra}$ and $R_{grc}$ are isomorphic as Hopf algebras. 
Then $R$ is a Nichols algebra.
\end{proposition}
\begin{proof}
Assume that the two Hopf algebras are isomorphic.
Write $R_{gra} = \bigoplus_{i=0}^n A_i$ where $A_i = J^i/J^{i+1}$ (for the sake of simplicity, we use here positive grading instead of the negative grading from Section \ref{sec:filtrations} for the radical filtrations). Then the algebra $R_{gra}$ is generated by $A_1$. The object $A_1$ is primitive, since $\Delta(A_1)\subseteq A_1\ot A_0 + A_0\ot A_1$ and $A_0\cong R/J\cong \one$. 
This implies, in particular, that $R_{gra}$ is generated by the subobject of primitives.
The algebra $R_{grc} = \oplus_{i=0}^m C_i$ has all primitives in degree 1. Since it is isomorphic to the algebra $R_{gra}$, it is also generated by $C_1$. This implies that $R_{grc}$ is a graded Hopf algebra which is generated in degree 1 and has all its primitive elements in degree 1. By Definition \ref{def:Nichols} this implies that $R_{grc}$ is a Nichols algebra. It follows that $R$ is generated by primitive elements.
By a dual argument, and by using the fact that $(R_{gra})^*\cong (R^*)_{grc}$ and $(R_{grc})^*\cong (R^*)_{gra}$ we get that $R^*$ is also generated by its primitive subobject. Lemma \ref{lem:pgen} gives us the desired result.
\end{proof}

This finishes the proof that if $\Ow_R$ is closed then $R$ is a Nichols algebra, because the fact that $R\lt R_{grc}$ and $R\lt R_{gra}$ together with the closure of $\Ow_R$ implies that $R_{grc}\cong R\cong R_{gra}$. 
Next, we will show that if $R$ is a Nichols algebra then $\Ow_R$ is closed.
Assume that this is not the case and let $R=\B(V)$ be a Nichols algebra with a non-closed orbit. The closure $\wb{\Ow_R}$ is the union of the orbit of $R$ with orbits of smaller dimension.
An orbit of minimal dimension in $\wb{\Ow_R}$ is closed. It follows that $R\lt R'$ for some $R'\in X_B$ with $\Ow_{R'}$ closed. 
But we already know that this implies that $R'$ is a Nichols algebra. 

Write $R'=\B(V')$. Then we have $\B(V)\lt \B(V')$. 
By \ref{lem:subobj} we know that this implies that $\Pp(\B(V))=V$ is isomorphic to a subobject of $\Pp(\B(V'))=V'$. 
Since $R\ncong R'$ the object $V$ must be isomorphic to a proper subobject of $V'$. Write $V'=V\oplus V''$. 
The split inclusion of objects in $\C$, $V\to V'\to V$, induces a split inclusion of Nichols algebras $\B(V)\to \B(V')\to \B(V)$. 
This implies that $\B(V)\ncong \B(V')$ because then $\B(V')$ properly contains $\B(V)$, and this contradicts the fact that $R$ and $R'$ are isomorphic to the same object of $\C$.
This finishes the proof of Theorem \ref{thm:main1}

Recall that $\Ga=\Aut_{\C}(B)$. The proof of Theorem \ref{thm:main1} gives us a description of the irreducible components of $X_B$:
\begin{theorem}
Let $B\in \C$. For every $V\in \C$ such that $\B(V)\in X_B$ write $X_V= \{R| R\in X_B\text{ and } R \lt \B(V)\}$. 
Then the subsets $X_V$ are stable under the action of $\Ga$ and are exactly the connected components of $X_B$. 
\end{theorem}
\begin{proof}
The fact that $X_V$ is stable under $\Ga$ is immediate.
Notice that for dimension considerations the number of objects $V$ such that $\B(V)\in X_B$ is finite.
We denote these objects by $V_1,V_2,\ldots V_m$.
We claim now that the dimension of the invariant subalgebra $K[X_B]^{\Ga}$ is finite.
Indeed, if $R\lt \B(V)$ then continuity considerations imply that $f(R)=f(\B(V))$ for every $f\in K[X_B]^{\Ga}$.
Consider the following homomorphism of algebras \begin{equation}\phi: K[X_B]^{\Ga}\to K^m\end{equation} $$f\mapsto (f(\B(V_1)),\ldots, f(\B(V_m))),$$ where $K^m$ is an algebra by the operations of pointwise addition and multiplication.  
If $f\in\Ker(\phi)$ then $f$ vanishes on the orbit of every Nichols algebra, and since $f$ is invariant under the action of $\Ga$, the continuity  of $f$ implies that it vanishes on every $\Ga$-orbit, so $f=0$. This implies that $\phi$ is injective. 
Since $\B(V_i)\ncong \B(V_j)$ when $i\neq j$ and $\Ow_{\B(V_i)}$ and $\Ow_{\B(V_j)}$ are closed and disjoint Proposition \ref{prop:sepinvs} implies that for every $i\neq j$ there is a function $f_{ij}\in K[X_B]^{\Ga}$ such that $f_{ij}(\B(V_i))\neq f_{ij}(\B(V_j))$. Since $\phi$ is an algebra map the image of $\phi$ is a unital subalgebra of $K^m$. The only unital subalgebra of $K^m$ that separates the points $\{1,\ldots, m\}$ is $K^m$ itself, and so $\phi$ is surjective and an isomorphism. 
We thus see that all the $X_{V_i}$ are closed, since $X_{V_i}$ is the zero set of the polynomial $1-\phi^{-1}(e_i)$ (where $\{e_i\}$ is the standard basis of $K^m$).

We next claim that $X_{V_i}$ is connected for every $i$. Indeed, assume that $X_{V_i}= Y_1\sqcup Y_2$ with $Y_1$ and $Y_2$ closed and nonempty. 
Take $y_1\in Y_1$. Then $\Ow_{y_1}=\Ga\cdot y_1$ is connected, contained in $X_{V_i}$ and intersects $Y_1$, so $\Ow_{y_1}\subseteq Y_1$. But then $\Ow_{\B(V_i)}\subseteq \overline{\Ow_{y_1}}\subseteq Y_1$. By a similar argument $\Ow_{\B(V_i)}\subseteq Y_2$ and this is a contradiction.
\end{proof}

Due to the last theorem, we can focus our attention on the different subvarieties $X_V$. These subvarieties are stable under the action of $\Ga$. 
The conditions in Theorem \ref{thm:main5} then boil down to the statement that if $\B(V)$ is finite dimensional, then the variety $X_V$ has a single orbit under the action of $\Ga$. 
We finish with a proof of Theorem \ref{thm:main4}:
\begin{proof}[Proof of Theorem \ref{thm:main4}]
Assume that $\B(V)$ is not rigid. In other words, assume that there are non-closed orbits in $X_V$. 
Take a non-closed orbit $\Ow_R$ of minimal dimension in $X_V$. We will prove that such an orbit is the orbit of a pre-Nichols algebra or a post-Nichols algebra.
For this consider the Hopf algebras $R_{gra}$ and $R_{grc}$. If both these Hopf algebras are isomorphic to $R$, then $R\cong R_{gra}\cong R_{grc}$, which implies that $R$ itself is a Nichols algebra by Proposition \ref{prop:gragrc}. This is a contradiction.
If both these algebras are non-isomorphic to $R$, then from the fact that $R_{grc},R_{gra}\in \wb{\Ow_R}$ it follows that the dimensions of the orbits $\Ow_{R_{gra}}$ and $\Ow_{R_{grc}}$ are smaller than the dimension of $\Ow_R$. By the minimality condition on $R$, this implies that $R_{grc}\cong R_{gra}\cong \B(V)$. But by Proposition \ref{prop:gragrc} again, this implies that $R$ itself is a Nichols algebra, which leads again to a contradiction.

We thus see that either $R\cong R_{gra}$ and $R\ncong R_{grc}$, or $R\ncong R_{gra}$ and $R\cong R_{grc}$. Assume first that $R\cong R_{gra}$ and $R\ncong R_{grc}$. 
Since $R\cong R_{gra}$, $R$ has a grading $R= \oplus R_i$ such that the Jacobson radical $J$ satisfy $J^i = \oplus_{j\geq i} R_j$ for every $i\geq 1$. The grading implies that $R_1$, which generate $R$ as an algebra, is a primitive object. But this already implies that $R$ is a pre-Nichols algebra, as it lies between $T(R_1)$ and $\B(R_1)$. Notice that it is impossible that $R_1\cong V$. Indeed, if this was the case then from dimension considerations the fact that $R$ projects onto $\B(R_1)$ would imply that $R\cong \B(R_1) = \B(V)$, contradicting the fact that the orbit of $R$ is not closed. 
By Lemma \ref{lem:subobj} it follows that $R_1$ is isomorphic to a subobject of $V$. We thus see that it must be a proper subobject. 

This shows that if $R\cong R_{gra}$ then $R$ is a pre-Nichols algebra.
If $R\cong R_{grc}$ then by duality of the radical and coradical filtrations we get that $R^*$ is a pre-Nichols algebra. This finishes the proof of Theorem \ref{thm:main4} and also of \ref{thm:main5}
\end{proof}
\end{section}

\begin{section}{Different notions of rigidity}\label{sec:rigidity}
In this paper we call a ccc Hopf algebra $R$ rigid if any Hopf algebra that specializes to it is isomorphic to it. 
There are other notions of rigidity, using deformations by a one-parameter family. We will explain here the relations between them. 

In \cite{AKM} and \cite{DCY} a deformation of a graded bialgebra (not necessarily a finite dimensional one) $B=\oplus_{i\geq 0} B_{(i)}$ by a parameter $\la$ is defined as a pair $(m_{\la},\Delta_{\la})$ such that \begin{equation}m_{\la} = \sum_{i=0}^{\infty} m_{(i)}\la^i\text{ and }\Delta_{\la} = \sum_{i=0}^{\infty} \Delta_{(i)}\la^i\end{equation} where $m_{(i)},\Delta_{(i)}$ are maps of degree $-i$, $m_{(0)}=m$ and $\Delta_{(0)}=\Delta$, and $(m_{\la},\Delta_{\la},u,\epsilon)$ defines a bialgebra structure on $B$ for every $\la$. (Bialgebra deformations of Hopf algebras are automatically Hopf algebras as well. Due to the uniqueness of the antipode, we do not need to consider it as part of the deformation data). 
It is shown that this is the same as a filtered Hopf algebra $U$ that satisfies $grU\cong B$. In \cite{AKM} and \cite{DCY} a graded Hopf algebra $B$ is called rigid if it has no non-trivial deformations. We will call it here deformation rigid. 
We claim the following:
\begin{lemma}
A finite dimensional Nichols algebra $\B(V)$ is rigid with respect to Definition \ref{def:rigidity} if and only if both $\B(V)$ and $\B(V^*)$ are deformation rigid.
\end{lemma}
\begin{proof}
Remark \ref{rmk:duality} implies that $\B(V)$ is rigid if and only if $\B(V^*)$ is rigid (with respect to Definition \ref{def:rigidity}).
To prove the first direction it will thus be enough to show that if $\B(V)$ is rigid with respect to Definition \ref{def:rigidity} then it is deformation rigid.

Assume that $(m_{\la},\Delta_{\la})$ is a one-parameter deformation of $\B(V)$. 
Using the grading $\B(V)=\oplus \B(V)_i$ we have a one-parameter family $\phi:\Gg_m\to \Aut_{\C}(B)$ which sends $\la$ to the automorphism which acts on $\B(V)_i$ by the scalar $\la^{-i}$. We then have that 
\begin{equation}(m_{\la},\Delta_{\la}) = \phi(\la)\cdot (m_1,\Delta_1)\text{ and }(m_0,\Delta_0)=\lim_{\la\to 0}\la\cdot (m_1,\Delta_1).\end{equation}
Since $\B(V)$ is rigid, any algebra specializing to it is isomorphic to it. So we get a Hopf algebra isomorphism $\Psi:\B(V)\cong (B,m_1,\Delta_1,u,\epsilon,S_1)$ where $S_1$ is the uniquely defined antipode. 

Next, we claim that $\B(V)_1$ consists of primitive elements with respect to the comultiplication $\Delta_1$ (and therefore, with respect to $\Delta_{\la}$ for every $\la\in K^{\times}$ as well). Indeed, by grading consideration we have that $\Delta_1|_{\B(V)_1} = \Delta{(0)} + \Delta_{(1)}$, where $\Delta_{(0)}=\Delta$. The map $\Delta_{(1)}:\B(V)_1\to \B(V)_0\ot \B(V)_0=\one\ot\one= \one$ must be zero since otherwise this will contradict the fact that $\epsilon(\B(V)_1)=0$. 
This implies that the isomorphism $\Psi$ will map $\B(V)_0$ to $\B(V)_0$ and $\B(V)_1$ to 
$\B(V)_1$. Without loss of generality we can assume that $\Psi|_{\B(V)_1} = Id_{\B(V)_1}$
(recall that both $\B(V)$ and the deformed algebra have the same underlying object $B$).
Since $\B(V)_n$ is the image of $\B(V)_1^{\ot n}\to \B(V)$ with respect to the iterated multiplication $m$, and since $\bigoplus_{i\leq n}\B(V)_i$ contains the image of $\B(V)_1^{\ot n}\to \B(V)$ under the iterated multiplication $m_1$, we get that $\Psi$ preserves the filtration $B_{(0)}\subseteq \B_{(1)}\subseteq\cdots $ of $B$ given by $\B_{(n)} = \bigoplus_{i\leq n}\B(V)_i$. But this already implies that the deformation $(m_{\la},\Delta_{\la})$ is trivial with respect to the definition in Section 2.3 of \cite{AKM}.

In the other direction, assume that $\B(V)$ and $\B(V^*)$ are deformation rigid. 
By Theorem \ref{thm:main5} we know that if $\B(V)$ is not rigid then there is a pre-Nichols Hopf algebra $R$ such that $R\lt \B(V)$ or $R\lt \B(V^*)$. Since both $\B(V)$ and $\B(V^*)$ are deformation rigid, we can assume without loss of generality that the first case holds.
The algebra $R$ is generated by its primitive elements, and therefore admits a filtration $$R_{(0)}\subseteq R_{(1)}\subseteq\cdots $$ where $R_{(n)} = \sum_{i=0}^n \Pp(R)^i$. This can easily be seen to be a Hopf algebra filtration of $R$. 
The associated graded object is then $\B(V)$ with its usual grading. Since we assumed that $\B(V)$ is deformation rigid we get that $R\cong \B(V)$, which is a contradiction.
\end{proof}

We recall here Theorem 6.2 from \cite{AKM}. This result relies on the classification result from \cite{A1}.
\begin{theorem} Assume that $\B(V)$ is finite dimensional and that the braiding on $V$ is of diagonal type. Then $\B(V)$ is deformation rigid.
\end{theorem}
This theorem, combined with Theorem \ref{thm:main5} gives a new proof for the generation of pointed Hopf algebras with abelian groups of group-like elements by group-like elements and skew-primitives.
Indeed, the above theorem implies that all the finite dimensional Nichols algebra in $\GGYD$ for $G$ abelian are rigid, and therefore by Theorem \ref{thm:main5} all finite dimensional ccc Hopf algebras in $\GGYD$ are Nichols algebras, and every Hopf algebras $H$ such that $H_0=KG$ with $G$ abelian is generated by group-like elements and skew-primitives.

\end{section}
\begin{section}*{Acknowledgments}
I would like to thank Nicol{\'a}s Andruskiewitsch, Iv{\'a}n Angiono and Istvan Heckenberger for their useful comments.
\end{section}

\end{document}